\documentclass[a4paper,12pt]{amsart}
\usepackage{times,a4wide,mathrsfs,amsthm,amssymb}
\usepackage[cal=boondoxo]{mathalfa} % mathcal
\usepackage{hyperref}
\usepackage{color}
\usepackage[all]{xy}

\newcommand{\C}{\mathbf{C}}
\newcommand{\Lef}{\mathbf{L}}
\newcommand{\ZZ}{\mathbf{Z}}

\newcommand{\QQ}{\mathbf{Q}}

\newcommand{\PP}{\mathbf{P}}

\newcommand{\LL}{\mathcal L}
\newcommand{\OO}{\mathcal O}
\newcommand{\Ss}{\mathcal S}
\newcommand{\DD}{\mathcal D}
\newcommand{\XX}{\mathcal X}
\newcommand{\YY}{\mathcal Y}
\newcommand{\TT}{\mathcal T}
\newcommand{\UU}{\mathcal U}
\newcommand{\VV}{\mathcal V}
\newcommand{\WW}{\mathcal W}

\newcommand{\Z}{\mathcal Z}

\newcommand{\im}{\operatorname{i}}

\newcommand{\codim}{\operatorname{\rm codim}}

\newcommand{\Tr}{\operatorname{\rm Tr}}

\newcommand{\wt}{\widetilde}
\newcommand{\wh}{\widehat}
\newcommand{\ima}{\operatorname{\rm Im}}
\newcommand{\rom}{\romannumeral}
\newcommand{\comp}{\raise1pt\hbox{{$\scriptscriptstyle\circ$}}}
\newcommand{\Mot}{\operatorname{\rm Mot}}
\newcommand{\Hom}{\operatorname{\rm Hom}}

\def\Corr{\mathop{\rm Corr}\nolimits}
\def\Gr{\mathop{\rm Gr}\nolimits}

\def\id{\mathop{\rm id}\nolimits}
\def\Nhat{{\widehat{N}}}
\def\sett#1#2{\{#1\mid#2\}}
\def\End{\operatorname{\rm End}}
\def\km#1{\mathop{\rm Km}(#1)}
\def\half{\frac 12}
\def\im{\operatorname{\rm Im}}

\def\mod2#1{{\color{blue}#1}}
 
\def\half{\frac 12}

\newtheorem{theorem}{Theorem}[section]
\newtheorem*{thm}{Theorem}

\newtheorem{lemma}[theorem]{Lemma}
\newtheorem{corollary}[theorem]{Corollary}
\newtheorem{proposition}[theorem]{Proposition}

\theoremstyle{remark}

\newtheorem*{remark*}{Remark}
\newtheorem*{facts}{Facts}
\newtheorem{remark}[theorem]{Remark}
\theoremstyle{definition}
\newtheorem{definition}[theorem]{Definition}

\newtheorem*{notation*}{Notation}
\newtheorem{example}[theorem]{Example}

\newtheorem{nonumberingt}{Acknowledgements}

\def\mapright#1{\mathop{\vbox{\ialign{
                ##\crcr
    ${\scriptstyle\hfil\;\;#1\;\;\hfil}$\crcr
 \noalign{\kern-1pt\nointerlineskip}
    \rightarrowfill\crcr}}\;}}

\def\mapleft#1{\mathop{\vbox{\ialign{
                ##\crcr
    ${\scriptstyle\hfil\;\;#1\;\;\hfil}$\crcr
    \noalign{\kern-1pt\nointerlineskip}
    \leftarrowfill\crcr}}\;}}

\def\into{\hookrightarrow}
\def\onto{\twoheadrightarrow}

\begin{document}

\date{4 September  2017} 
\author[Robert Laterveer]
{Robert Laterveer}
\address{Institut de Recherche Math\'ematique Avanc\'ee, Universit\'e 
de Strasbourg, 7 Rue Ren\'e Des\-car\-tes, 67084 Strasbourg CEDEX, France.}
\email{robert.laterveer@math.unistra.fr}

\author[Jan Nagel]
{Jan Nagel}
\address{Institut de Mat\'ematiques de Bourgogne,
9 avenue Alain Savary,
21000 DIJON CEDEX, France.}
\email{johannes.Nagel@u-bourgogne.fr}

\author[Chris Peters]
{Chris Peters}
\address{Discrete Mathematics, Technische Universiteit Eindhoven,
Postbus 513,
5600 MB Eindhoven, Netherlands.}
\email{c.a.m.peters@tue.nl}

\title[Complete intersections in varieties with finite dimensional motive]{On complete intersections in varieties\\ with finite-dimensional motive}

\begin{abstract} Let $X$ be a complete intersection inside a variety $M$ with finite dimensional motive and for which the Lefschetz-type conjecture $B(M)$ holds.
We show how conditions on the niveau filtration on the homology of $X$ influence directly the niveau %
on the level 
of Chow groups. This leads to a 
generalization of Voisin's result. The latter states that \emph{if $M$ has trivial Chow groups and if  $X$ has non-trivial variable cohomology  parametrized by $c$-dimensional algebraic cycles, then the cycle class maps $A_k(X) \to H_{2k}(X)$ are injective for $k<c$}. We give variants involving group actions which lead to several  new examples with finite dimensional Chow motives.
\end{abstract}

\keywords{Algebraic cycles, Chow groups, Chow-K\"unneth decompositions, finite-dimensional motives, (co)niveau filtration}

\subjclass{14C15, 14C25, 14C30.}

\maketitle

\section{Introduction}

\subsection{Background} 
Let $X$ be a smooth, complex projective variety of dimension $d$. While the cohomology ring\footnote{See the conventions about the notation at the end of the introduction.} $H^*(X)$ is well
understood, this is far from true for the Chow ring $A^*(X)$, the ring  of  algebraic cycles on $X$ modulo rational equivalence. The two
are linked through the \emph{cycle class map} 
\[
A^*(X)  \to  H^{2*}(X), \quad  \gamma \mapsto [\gamma].
\]
If this map is injective we say that \emph{$X$ has trivial  Chow  groups}.
If this is not the case, the kernel  $A^*_{\rm hom}(X)$, the "homologically trivial" cycles,  then can be investigated  through the \emph{Abel-Jacobi map}
\[
A^*_{\rm hom}(X) \to J^*(X)
\]  
with kernel $A^*_{\rm AJ}(X)$, the "Abel-Jacobi trivial" cycles.  If $X$ is a curve, Abel's theorem tells us that
 %classical theory of cycles tells us that 
$A^1_{\rm AJ}(X) = 0$. 
%(the $0$-cycles of degree zero form $J^1(X)$, the Jacobian of $X$).

The interplay between Hodge theoretic aspects of cohomology and cycles became apparent through the fundamental work of
Bloch and Srinivas \cite{BS} as complemented by \cite{small,connect}.  
They investigate the consequences for the Chow groups and cohomology groups of $X$ if      the  class $\delta \in A^d(X\times X)$ of the  diagonal  $\Delta \subset X\times X$ 
admits a decomposition into summands 
having support on lower dimensional varieties.
This clarifies the role   of the so-called
\emph{coniveau filtration}  $N^\bullet H^*(X)$ in cohomology which takes care of cycle classes supported on varieties of varying dimensions.
Charles Vial \cite{V4} discovered a variant which works better in homology which he called the \emph{niveau filtration} $\wt{N}^\bullet H_*(X)$. We introduce a 
\emph{refined niveau filtration} on homology $
\wh{N}^\bullet H_*(X)$ which is compatible with polarizations. The precise definitions are  given below in Sect.~\ref{sec:NivandConiv}.  
Suffices to say that we have inclusions $\wh{N} ^\bullet H_*(X)\subseteq \wt{N}^\bullet H_*(X)\subseteq N^\bullet H_*(X)$
 with equality everywhere if  the Lefschetz conjecture $B$ is true for all varieties. Conjecture $B$ is  recalled below in Section~\ref{sec:ConjB}.

Note that  the K\"unneth formula  $ \delta= \sum_{k=0}^{2d}  \pi_k $,  with  $\pi_k\in H^{2d-k}(X)\otimes H^k(X)=  H^k(X)^*  \otimes H^k(X)  $, 
can be interpreted as an identity inside the ring of endomorphisms of $H^*(X)$.  
Since  $\delta \in H^{2d}(X\times X)$ acts as the identity on $H^*(X)$, in $\End H^*(X)$ one thus obtains 
 the (cohomological)   \emph{K\"unneth-decomposition}
\[
\id = \sum_{k=1}^{2d}  \pi_k, \quad   \pi_k \in \End H^*(X)  \text{ a projector with  } \pi_k|_ {H^j (X)}= \delta_{jk}\cdot  \id.
\] 
The projectors are mutually orthogonal, that is $\pi_j\comp \pi_k =0$ if $j\not=k$. Moreover,
the K\"unneth decomposition  is by construction compatible with Poincar\'e duality and so  is called \emph{self-dual}; in other words $\pi_k $
is the transpose of $\pi_{2d-k}$ for all $k<d$.

Even if the \emph{K\"unneth components} $\pi_k$ are classes of  algebraic cycles,  their sum need not give a decomposition of the diagonal.
If this is the case, and if, moreover,  these give  a   self-dual 
decomposition of the identity in $\End A^*(X)$ by mutually orthogonal projectors,
one speaks of a (self-dual)  \emph{Chow--K\"unneth decomposition}, abbreviated as "CK-decomposition". Its existence has been conjectured by Murre 
\cite{Mur}, and it has been established in low dimensions  and a few other cases.

One would like to have a refined CK-decomposition which takes into account the coniveau filtration
or the (refined)  niveau filtration, since then  the conclusions of \cite{BS} et. al. can be applied. This is related to the validity of the standard conjecture  $B(X)$
as   reviewed in Section~\ref{sec:ConjB}.

\subsection{Set up and results}

Following  Voisin   \cite{V0,V1},  we consider complete intersections $X$   of dimension $d$ inside a given  smooth complex variety projective 
variety $M$ and we ask about the relations between the Chow groups of $M$ and $X$. 
On the level of cohomology this is a consequence of the classical Lefschetz theorems:  apart from the "middle" cohomology $H^d(X)$ the cohomology of $X$ is completely determined by $H^*(M)$, while  for the middle cohomology one has a direct sum splitting 
\[
H^d(X)= H^d_{\rm fix}(X) \oplus H^d_{\rm var} (X)
\]
into fixed cohomology $H^d_{\rm fix}(X)= i^*H^d(M)$ and its orthogonal complement  $H^d_{\rm var} (X)$ under the cupproduct pairing. Here $i:X\into M$
is the inclusion, and $i^*:H^d(M)\to H^d(X)$ is injective.

For this to  have  consequences on the level of  Chow groups, it seems natural to assume that  $M$ has trivial Chow groups. 
This is the point of view of Voisin in \cite{V1}.  Her main result uses the notion of  a  subspace $H\subset H^k(X)$ "being  parametrized by $c$-dimensional algebraic cycles"
\cite[Def. 0.3]{V1} which is slightly stronger than  demanding that  $H\subset \wh{N}^c H^k(X)$, where $\hat N$ is our refined version of Vial's filtration.
%Our  filtration is, we  recall,    compatible with polarizations and captures Voisin's notion, but is  indeed  slightly weaker. 
A comparison of our filtration  with Vial's is given in Section~\ref{sec:ModNivFIlts}. See in particular
Remark~\ref{rmk:Comparison}. 
We can now state  Voisin's main result from \cite{V1}:

 \begin{thm} Assume that $M$ has trivial Chow groups and that $X$ has non-trivial variable cohomology  parametrized by $c$-dimensional algebraic cycles.
 Then the cycle class maps $A_k(X) \to H_{2k}(X)$ are injective for $k<c$.
 \end{thm}

Our idea is to replace the condition of  $M$ having trivial Chow groups  by finite dimensionality of the motive of $M$ --  which  conjecturally  is true for all varieties.
\footnote{See \cite{MNP} for background on Chow motives.} The  main idea  which makes this operational is the following nilpotency result (=Theorem~\ref{weaknilp}): 
 if $r$ is the codimension of $X$ in $M$, a degree $r$ correspondence which restricts   to a cohomologically trivial degree zero correspondence on $X$
is nilpotent as a correspondence on $X$. 

The second ingredient is  due to Voisin  \cite[Proposition 1.6]{V0}:  a degree $d$ \emph{cohomogically} trivial 
relative correspondence   can be modified in a controlled way such  that  the new relative 
correspondence is fiberwise   \emph{rationally} equivalent to zero.

Given these inputs, the argument leading to our results now runs as follows. First we 
make  use of the refined niveau filtration   by way of Propositions~\ref{factorisation} and \ref{hat-tilde}
to find relative correspondences that decompose
 the diagonal in \emph{homology} in the way we want. To the difference we apply  the Voisin result.
This provides  first of all information on the level of the \emph{Chow groups of the fibers} and, secondly, allows us to  apply the nilpotency result. 
Writing this out gives  strong variants  of  the above theorem of Voisin.   These have been phrased  in homology rather than cohomology because, as 
mentioned before, Vial's filtration and ours behave better in the  homological setting. 
One of our main results can be paraphrased as follows.
%  -- in this language  "vanishing homology"  corresponds to variable cohomology:

%\begin{thm}[=Theorem~\ref{main2}]   
%Let $i: X\into  M$ be a complete intersection of dimension $d$. Suppose that
%\begin{enumerate}
%\item  $B(M)$ holds;
%\item  The Chow motive of $M$  is finite dimensional; 
%\item \modif{ $H_k(M) = N^{[{k\over 2}]}H_kM)$ for $k\le d $;
%\item $H_d^{\rm var}(X)\ne 0$ and for some positive integer $c$ we have $H_d^{\rm var}(X)%\subset{\hat N}^c H_d(X)$. 
%\end{enumerate}
%Then   for $k<c$ or for $k>d-c$ we have 
%\[
%  i^*: A_{k+r}^{\rm AJ}(M) \onto A_k^{\rm AJ}(X) ,\quad 
%                         i_* :A_k^{\rm AJ}(X) \into  A_k^{\rm AJ}(M ).
%\]
%\end{thm}

\begin{thm}[=Theorem~\ref{main2}]   Suppose that $B(M)$ holds, that  the  Chow motive  of $M$  is finite-dimensional
and that $H_k(M) = {N}^{[{{k+1}\over 2}]}H_k(M)$ for $k\le d$.
Suppose  $H_d^{\rm var}(X)\ne 0$, and that for some positive integer $c$ we have $H_d^{\rm var}(X)\subset{\hat N}^c H_d(X)$. 
Then     $A_k^{\rm hom}(X) = 0$  if $k<c$ or $k>d-c$.
  \end{thm}

Voisin's result  is  a  direct consequence:  by  \cite[Theorem 5]{V2} varieties with trivial Chow groups have finite dimensional motive and conjecture $B$ holds for them as well and the
 condition $H_k(M) = {N}^{[{{k+1}\over 2}]}H_k(M)$ holds since $M$ has trivial Chow groups.  
Surprisingly, if we apply Vial's result  \cite{V}, we find that if the condition in the above theorem holds for $c=[\frac{d}{2}]$,  then $h(X)$ itself also has finite dimension and up to motives
of  curves and Tate twists is a direct factor of $h(M)$ (Corollary \ref{cor:fdm}). 

The known examples of finite dimensional motives are all directly related to curves,  which very much limits  the search for examples.
However, inside the realm of motives we can use other projectors besides  the identity, namely those that come from group actions.
In Section~\ref{sec:vars}, we have formulated variants of the main result involving  actions of a finite abelian group, say $G$. 
Then, even if the level of the Hodge-niveau filtration on variable cohomology  is too big to apply our main theorems, there might be  
 a $G$-character space  which has the correct Hodge-level. Provided
the (generalized) Hodge conjecture holds,  which is automatically the case in dimensions $\le 2$,  this then ensures  the  desired condition on  the niveau filtration.
In Section~\ref{sec:exmples}  we  construct  examples where this is the case and
for which one of the  group variants of the main theorem  can be successfully applied. These examples all yield  new finite dimensional motives because of the
above mentioned result of Vial.

We have given several  types of examples:  
\begin{itemize}
\item a threefold of general type with $p_g=q=0$,
\item hypersurfaces in abelian threefolds, including the Burniat-Inoue surfaces,
\item hypersurfaces in a product of a hyperelliptic curve and certain types
of K3 surfaces,
\item  hypersurfaces in  threefolds that are products of three curves, one of which is hyperelliptic,
\item  odd-dimensional complete  intersections of $4$ quadrics -- generalizing the  Bardelli example \cite{Bardelli}.
\end{itemize}

For simplicity we have only considered involutions
since then all invariants can easily  be calculated, but it will be clear that the method of construction  allows for many more  examples of varieties
admitting  all kinds of  finite abelian  groups of automorphisms. 

\begin{nonumberingt} We want to thank Claire Voisin, who kindly suggested the example of subsection \ref{sec:hypab3folds} to one of us. Thanks also to Claudio Pedrini, for helpful comments on an earlier version of this article.
\end{nonumberingt}

\begin{notation*} Varieties will be defined over $\C$ (except for Appendix B, where we consider algebraic varieties and motives over a field $k$). 
We use $H^*,H_*$ for  the (co)homology  groups
with $\QQ$-coefficients and likewise we write  $A^*, A_*$ for the Chow groups with $\QQ$--coefficents.\\
The category of Chow motives (over a field $k$) is denoted  by $\Mot_{\rm rat}(k)$, the category of {\em covariant} homological motives  by  $\Mot_{\rm hom}(k)$ and the category of numerical motives $\Mot_{\rm num}(k)$.
For a smooth projective manifold $X$, we let $h(X)\in \Mot_{\rm rat}(k)$ be its Chow motive.\\
We denote the integer part of a rational number $a$ by $[a]$.
\end{notation*}

%\newpage

\section{Preliminaries} \label{sec:prelim}

\subsection{Correspondences} \label{sec:corr}

If $X$ and $Y$ are projective varieties with $X$ irreducible of dimension $d_X$,  a correspondence of  degree $p$ is an element of
$$
\Corr_p(X,Y) := A_{d_X+p}(X\times Y).
$$
A degree $p$ correspondence $\gamma$ induces maps
$$
\gamma_*:A_k(X)\to A_{k+p}(Y),\ \ \gamma_*:H_k(X)\to H_{k+2p}(Y).
$$
If, moreover,  $X$ and $Y$ are smooth projective, we have correspondences of cohomological degree $p$, i.e., elements  
 $$
\gamma\in \Corr^p(Y,X) := A^{d_Y+p}(Y\times X),
$$
which induce 
$$
\gamma^*:A^k(Y)\to A^{k+p}(X),\ \ \gamma^*:H^k(Y)\to H^{k+2p}(X).
$$

\begin{definition}  \label{dfn:Factoring}
Let $\gamma\in\Corr_p(X,X) = A_{d+p}(X\times X)$  be a self-correspondence of degree $p$ where $d=d_X$.
\begin{enumerate}
\item
Let $Z$ be smooth and equi-dimensional. We say that $\gamma$  \emph{factors  through  $Z$ with shift $i$}  if there exist correspondences $\alpha\in\Corr_{i}Z,X)$ and $\beta\in\Corr_{-j}(X,Z)$ ($i-j=p$) such that  $\gamma=\alpha\comp \beta$  and  $d-(i+j)=\dim Z$.
\item
We say that $\gamma$ is \emph{supported on} $V\times W$ if 
$$
\gamma\in\ima{( A_{d+p}(V\times W)\mapright{(i\times j)_*}A_{d+p}(X\times X)})
$$
where $i:V\to X$ and $j:W\to X$ are inclusions of subvarieties of $X$.
\end{enumerate}
\end{definition}
The usefulness of these concepts  follows from the following evident results.
\begin{lemma}  
\begin{enumerate}
\item If a correspondence $\gamma\in\Corr_0(X,X)$ factors through $Z$ with shift $c$, then  $\gamma$  and $^t\gamma$ act  trivially on $A_j(X)$ for $j<c$ or $j>d-c$.
\item If a correspondence $\gamma\in\Corr_0(X,X)$ is supported on $V\times W\subset X\times X$, then $\gamma$ acts trivially on $A_j(X)$ for $j< \codim V$ or $j> \dim W$ 
and ${^t} \gamma$  acts trivially on $A_j(X)$ for $j< \codim W$ or $j> \dim V$.
\end{enumerate}
\label{lem:Factoring}
\end{lemma}
 
%Note that by Lieberman's Lemma (REF), $\gamma\in\Corr_p(X,X) =A_{n+p}(X\times X)$ factors through a variety $Z$ of dimension $d$ if and only if
%$$
%\gamma = ({^t \beta}\times\alpha)_*(\Delta_Z)\in\ima{A_d(Z\times Z)\to A_{n+p}(X\times X)}.
%$$

\subsection{Standard conjecture $B(X)$} \label{sec:ConjB}

Let $X$ be a smooth complex projective variety of dimension $d$, and $h\in H^2(X)$ the class of an ample line bundle. The hard Lefschetz theorem asserts that the map
  \[  L^{d-k}_X\colon H_{2d-k}(X)\to H_{k}(X)\]
  obtained by cap product with $h^{d-k}$ is an isomorphism for all $k< d$. One of the standard conjectures asserts that the inverse isomorphism is algebraic:

\begin{definition} Given a variety $X$, we say that \emph{$B_k(X)$ holds} if the isomorphism 
  \[  
  \Lambda^{d-k}= (L^{d-k})^{-1}\colon 
  H_k(X)\stackrel{\cong}{\rightarrow} H_{2d-k}(X)
  \]
  is induced by a correspondence. We say that the \emph{Lefschetz standard conjecture $B(X)$ holds} if $B_k(X)$ holds for all $k<d$.
 \end{definition}

 \begin{remark} \label{Bholds} 
 The Lefschetz (1,1) theorem implies that $B_k(X)$ holds if $k\le 1$ and hence it holds for  curves and surfaces.
 It is stable under products and hyperplane sections \cite{K0,K}  and so, in particular, it is true for complete intersections in products of projective spaces.
 It is known that $B(X)$ moreover holds for the following varieties: 
\begin{itemize} 
\item   abelian varieties \cite{K0,K};
\item  threefolds not of general type \cite{Tan};
\item  hyperk\"ahler varieties of 
 $K3^{[n]}$-type \cite{ChM};
 \item Fano varieties of lines on cubic hypersurfaces \cite[Corollary 6]{Lat}; %\footnote{(Robert) Dit geval nog toegevoegd}
 \item  $d$-dimensional varieties $X$ which have $A_k(X)_{}$ supported on a subvariety of dimension $k+2$ for all $k\le{d-3\over 2}$ \cite[Theorem 7.1]{V};
 \item  $d$-dimensional varieties $X$ which have $H_k(X)=N^{[{k\over 2}]}H_k(X)$ for all $k>d$ \cite[Theorem 4.2]{V2}.
\end{itemize}
 \end{remark}
Below we  shall use  the following  well  known implication  of  $B(X)$.
\begin{proposition}[\protect{\cite[Thm. 2.9]{K0} }]  \label{KunnIsAlg} Suppose that $B(X)$ holds. Then 
 the K\"unneth projectors are algebraic, i.e., there exist correspondences $\pi_k\in\Corr_0(X,X)$ such that $\pi_{k*} \mid_{H_j(X)} = \delta_{kj}.\id$ and $\Delta_X\sim_{\rm hom}\sum_k \pi_k$. \\
\end{proposition}
Refinements will be stated below in Section~\ref{sec:NivandConiv}.

\subsection{Finite dimensional motives and nilpotence} \label{sec:FinDImMots}

We refer to    \cite{An}, \cite{Iv},  \cite{Kim},  \cite{MNP} for the definition of a Chow motive and its dimension. We also need the concept of  a \emph{motive  of abelian type},
by definition   a  Chow  motive $M$ for which some twist $M(n)$ is a direct summand of the motive of a product of curves.

A crucial property of varieties with finite-dimensional motive is the nilpotence theorem.

\begin{theorem}[Kimura \cite{Kim}]\label{nilp} Let $X$ be a smooth projective variety with finite-dimensional motive. Let  
$\Gamma\in 
\Corr_0(X,X)$ be a correspondence which is numerically trivial. Then there exists a nonnegative integer $N$ such that
     $\Gamma^{\comp  N}=0$ in $ \Corr_0(X,X)$.
\end{theorem}

 Actually, the nilpotence property (for all powers of $X$) could serve as an alternative definition of finite-dimensional motive, as shown by a result of Jannsen \cite[Corollary 3.9]{J4}.
   Conjecturally, any variety has finite-dimensional motive \cite{Kim}. We are still far from knowing this, but at least there are quite a few non-trivial examples:

\begin{remark}  \label{findinmots}
The following varieties are known to have a  finite-dimensional motive: 
\begin{itemize}
\item varieties dominated by products of curves \cite{Kim} as well as  varieties of dimension $\le 3$ 
rationally dominated by products of curves \cite[Example 3.15]{V3};
\item  $K3$ surfaces with Picard number $19$ or $20$ \cite{P};
\item surfaces not of general type with vanishing geometric genus \cite[Theorem 2.11]{GP} as well as 
many examples of surfaces of general type with $p_g=0$ \cite{PW,V8};
\item  Hilbert schemes of surfaces known to have finite-dimensional motive \cite{CM};
\item Fano varieties of lines in smooth cubic threefolds, and Fano varieties of lines in smooth cubic fivefolds \cite{fanocubic};
\item  generalized Kummer varieties \cite[Remark 2.9(\rom2)]{Xu};
\item   $3$-folds with nef tangent bundle   \cite[Example 3.16]{V3}),  
%and $4$-folds with nef tangent bundle \cite{Iy2}
as well as  certain 3-folds of general type \cite[Section 8]{V5};
%\item  log-homogeneous varieties in the sense of \cite{Br} (this follows from \cite[Theorem 4.4]{Iy2});
\item  varieties $X$ with Abel-Jacobi trivial Chow groups (i.e. $A^k_{AJ}X_{}=0$   for all $k$)  \cite[Theorem 4]{V2};
\item products of varieties with finite-dimensional motive \cite{Kim}.
\end{itemize}
 \end{remark}

\begin{remark*}
It is worth pointing out that up till now, all examples of finite-dimensional Chow motives happen to be  of abelian type. On the other hand, ``many'' motives are known to lie outside this subcategory, e.g. the motive of a general hypersurface in $\PP^3$ \cite[Remark 2.34]{Ay}.
\end{remark*}

The following result is a kind of ``weak nilpotence'' for subvarieties of a variety $M$ with finite-dimensional motive; any correspondence that comes from $M$ and is numerically trivial turns out to be nilpotent.

\begin{proposition}\label{weaknilp} 
Let $M$ be a smooth projective variety with finite-dimensional Chow motive
and let  $X\subset M$ be a smooth projective subvariety of   codimension  $r$.
For any correspondence  $\Gamma\in\Corr_{r}(M,M)$ with the property that  the restriction
  \[ 
  \Gamma\vert_X\ \ \in\Corr_0(X,X) 
   \]
  is homologically trivial,  there exists a nonnegative integer $N$ such that
  \[  
  (\Gamma\vert_X)^{\comp  N}=0 \text   {  in   }   \Corr_0(X,X) .
  \]
  \end{proposition}

\begin{proof} %We denote by $\Corr^p(M,M) = A^{\dim M + p}(M\times M)$ the group of correspondences of degree $p$ from $M$ to $M$. Let $c$ be the codimension of $X$ in $M$. 
%Note that $\Gamma\in\Corr^{-c}(M,M)$. 
Put $L = i_*\comp  i^*\in\Corr_{-c}(M,M)$ and $T=\Gamma\comp  L\in\Corr_0(M,M)$. We have 
$$
\Gamma\vert_X = (i\times i)^*(\Gamma) = i^*\comp \Gamma\comp  i_*.
$$
By induction on $k$ one shows that
\begin{equation}
\label{formula}
\Gamma\vert_X^{k+1} = i^*\comp  T^k\comp \Gamma\comp  i_*
\end{equation}
for all $k\ge 0$. As 
\begin{eqnarray*}
T^2 & = & \Gamma\comp  i_*\comp  i^*\comp \Gamma\comp  i_*\comp  i^* \\
& = & \Gamma\comp  i_*\comp \Gamma_X\comp  i^*,
\end{eqnarray*}
$T^2$ is homologically trivial. Hence $T^2$ is nilpotent by \cite{Kim}, say $T^{2\ell}=0$. Hence $\Gamma_X$ is nilpotent of index $N = 2\ell+1$ by (\ref{formula}).
\end{proof}

\subsection{Coniveau and niveau filtration} \label{sec:NivandConiv}

\begin{definition}[Coniveau filtration \cite{BO}]
\label{con}
Let $X$ be a smooth projective variety of dimension $d$.
 The $j$-th level  of the \emph{coniveau filtration} on  cohomology (with $\QQ$-coefficients) 
is defined as the subspace generated by the classes supported on subvarieties $Z$ of dimension $\le d-j$:
  \[
  N^j H^k(X)= \sum_Z \ima\bigl(i_*:H^{k}_Z(X)\to H^k(X)\bigr). 
   \]
  This gives a decreasing filtration  on $H^k(X)$.
  We may instead use  smooth varieties $Y$  of dimension exactly $d-j$ provided we use  degree $j$  correspondences from $Y$ to $X$: such a correspondence sends $Y$ to a cycle $Z$ of dimension $\le d-j$ in $X$ and all cycles can be obtained in this way. When we rewrite this  in terms of homology we get  
\[ 
N^j H_k(X) =\sum _{Y,\gamma} \ima \bigl(\gamma_*: H_{k}(Y)\to H_k(X)\bigr),
                                                  \] 
where $Y$ is smooth projective of dimension  $ k -j$ and $\gamma\in \Corr_0(Y, X)$.
   \end{definition}

Since  the $j$-th level of the filtration consists of  the classes supported  on varieties of dimension $k-j$,
the filtration stops beyond $k/2$: a variety of dimension $<k/2$ has no homology in
 degrees $\ge k$:
\[
 0 =N^{[{k\over 2}]+1} H_kX \subset N^{[{k\over 2}]} H_k(X) \subset \cdots \subset N^1 H_k(X)\subset N^0H_k(X)=H_k(X).
\]

\begin{remark*} Under Poincar\'e duality  one has  an identification  $N^jH^k(X) = N^{d-k+j}H_{2d-k}(X)$.  
%In particular,$ N^{[{d\over 2}]} H^d(X)=  N^{[{d\over 2}]} H_d(X)$.

\end{remark*} 

Vial \cite{V4} introduced the following variant of the coniveau filtration:
\begin{definition}[Niveau filtration]
 Let $X$ be a smooth projective variety. The {\em niveau filtration} on homology is defined as
  \[ \wt{N}^j H_k(X)=\sum \ima\bigl( \gamma ^*: H_{k-  2j}(Z)\to H_k(X)\bigr)\ ,\]
  where the sum is taken  over all smooth projective varieties $Z$ of dimension $ k-  2j$, and all correspondences $\gamma\in \Corr_{j}(Z\times X)_{}$.
  \end{definition}

\medskip
\begin{remark}\label{is}
The idea behind this definition is that one should be able to lower the dimension of the variety $Y$ appearing in Definition \ref{con} using the Lefschetz standard 
conjecture. By Hard Lefschetz we have an isomorphism $\Lambda^j_Y:H_{k-2j}(Y)\mapright{\cong}H_k(Y)$ and by the Lefschetz hyperplane theorem a surjection 
$\iota_*:H_{k-2j}(Z)\to H_{k-2j}(Y)$ with $Z = Y\cap H_1\ldots\cap H_j$ a complete intersection of  $Y$ with $j$ general hyperplanes. Hence there is a surjective map 
$\iota_*\comp \Lambda^{j}_Y:H_{k-2j}(Z)\to H_k(Y)$ which is algebraic if $B_{k-2j}(Y)$ holds and thus  $N^j H_k(X)= \wt{N}^j H_k(X)$.

This discussion  also shows that
\begin{itemize}
\item  $\wt{N}^j H_k  (X)\subset N^j H_k (X)$ 
\item  $\wt{N}^j H_k (X)= N^j H_k(X$ if $k-2j \le 1$.
\end{itemize}
\end{remark}

\subsection{On variable  and fixed cohomology}
\label{sec:VarAndFixed}
Let $M$ be a smooth projective variety of dimension $d+r$ and $i:X \into M$ a smooth complete intersection of dimension $d$.
Let us assume $B(M)$ so that  the operator   $\Lambda^r$ on $H_*(M)$ is induced by an algebraic cycle  $\Lambda_M^r$ on $M\times M$.
Set 
\[
\pi^{\rm fix}(X) :=  i^* \Lambda_M^r i_*,\quad  \pi^{\rm var}(X)=\Delta - \pi^{\rm fix}(X).
\]
Recall that setting 
\[
\aligned
H_d^{\rm fix}(X) &=\im( i^* : H_{d+2r}(M) \to H_d(X)),\\
H_d^{\rm var}(X) &= \ker(i_*:H_d(X)\to H_d(M)),
\endaligned
\]
one has  a direct sum decomposition
\[ H_d(X)=
H^{\rm fix}_d(X)   \oplus H^{\rm var}_d(X),
\]
which is orthogonal with respect to the intersection product. We claim the following result.
\begin{lemma}  \label{lem:vanproj}  The operators $ \pi^{\rm fix}(X)$ and $\pi^{\rm var}(X)$ are  homological projectors which
give  the  projection  of the total cohomology onto $H^{\rm fix}(X)$, respectively $H^{\rm var}(X)=H^{\rm var}_d(X)$. %These projectors are  compatible with $Q_X$.
\end{lemma}
\proof   We first observe that $i_* : H_*(X) \to L^r  H_* (M)$ since $i_* H^{\rm fix}(X)= i_*\comp i^* H(M)= L^r H(M)$. 
 On the image of $L$ the two operators $L$ and $\Lambda$ are inverses. So, since\footnote{In fact this is only true up to a multiplicative constant
 but changing $\Lambda^r$ accordingly corrects this.}   $i^*\comp i_* =  L^r$, we find 
\[
\aligned
(i^*\comp \Lambda^r \comp  i_*)^2  &= i^* \comp \Lambda^r \comp i_* i^* \comp \Lambda^r \comp  i_*\\
 &=  i^* \comp    \Lambda^r  \comp L^r      \Lambda^r \comp  i_* \\
 &=  i^*\comp \Lambda^r \comp  i_*,
 \endaligned
\]
i.e. $\pi^{\rm fix}$ is indeed a projector, and so is  $ \pi^{\rm var}$.  These  projectors define   a splitting  on cohomology given by  
\[
z= i^*\Lambda^r  i_* z + (z-     i^*\Lambda^r    i_* z).
\]
On the image of $i_*$ the two operators $L$ and $\Lambda$ commute and are each others inverse and so
\[
\aligned
i_*  (z-     i^*\Lambda^r   i_* z) & = i_*z - L^r \Lambda ^r i_*z\\
&= i_*z -i_*z =0
\endaligned
\]
which shows that $\pi^{\rm var}$ indeed gives the projection onto variable  homology and so $\pi^{\rm fix}$ projects onto the fixed cohomology.
 \endproof
 
 \begin{remark} %hier ook 
 The degree zero correspondences $\pi^{\rm fix}$ and $\pi^{\rm var}$ are not necessarily projectors on the level of Chow groups,
 although one can show that finite-dimensionality of   $h(M)$ and $B(M)$ can be used to modify these correspondences in such a way that they become projectors.
 For what follows  we do not  need this.
 \end{remark}

\section{Niveau filtrations and  polarisations}
\subsection{Polarisations} \label{sec:Pols}

Recall that for $k\le d=\dim X$ we have the Lefschetz decomposition
%\footnote{ $L$ is cup product with the hyperplane class and $H_*^{\rm pr}$ stands for primitive homology (dual to primitive cohomology).} 
$$
H^k(X) = \oplus_r
L^r H^{k-2r}_{\rm pr}(X).
$$ 
Following \cite[p. 77]{Weil} we define a polarisation $Q_X$ on $H^k(X)$ as follows. Given $a$,$b\in H^k(X)$, write $a=\sum_r L^ra_r$, $b=\sum_r L^rb_r$ and define 
$$
\aligned
Q_X(a,b)  &= \sum_r (-1)^{{k(k-1)\over 2}+r}\langle L^{d-k+2r}a_r,b_r\rangle 
%&=   \sum_r (-1)^{{k(k-1)\over 2}+r}\langle a_r,L^{k-d+2r}b_r\rangle,
\endaligned
$$
where 
$$
\langle \, ,\rangle:H^{2d-k+2r}(X)\otimes H^{k-2r}(X)\to H^{2d}(X)\cong\QQ
$$
denotes the cup product. As the Lefschetz decomposition is $Q_X$-orthogonal, we can rewrite this in the following form. Let $p_r:H^k(X)\to L^rH^{k-2r}_{\rm pr}(X)$ be the projection, and define 
$$
s_X = \sum_r (-1)^{{k(k-1)\over 2}+r} L^{r}\comp  p_r.
$$
Then $Q_X(a,b) = \langle L^{d-k}(a),s_X (b) \rangle $.    

When we translate this to homology we obtain a polarisation $Q_X$ on $H_k(X)$ ($k\le d$) given by
$$
Q_X(a,b) = \langle a, \Lambda^{d-k}(s_X(b))\rangle
$$
where $s_X$ is (up to sign) the alternating sum of the projections
$p_r: H_k(X)\to L^rH^{\rm pr}_{k+2r}(X)$ to the primitive homology (dual to primitive cohomology).

\begin{lemma}
\label{sX alg}
If $B_{\ell}(X)$ holds for $\ell\le2 \dim X-k-2$ the operator $s_X \in\End(H_k(X))$ is algebraic.
\end{lemma}

\begin{proof} See \cite[Lemma 7]{Ch} or \cite[Lemma 1.7]{V4}
\end{proof}

\subsection{Modified niveau filtration} \label{sec:ModNivFIlts} 

We start by a discussion of adjoint correspondences. This material is treated from a cohomological point of view in \cite[section 4.2]{Lie Fu}. 

\begin{definition}
\label{adjoint}
Let $X$ and $Y$ be smooth projective varieties of dimension $d_X$, $d_Y$. Let $\gamma\in\Corr_j(X,Y)$. 
\begin{enumerate}
\item[(i)]
We say that $\gamma$ {\em admits a $k$-adjoint} if there exists $\gamma^{\rm adj}\in\Corr_{-j}(Y,X)$ such that
$$
Q_Y(\gamma_*(a),b) = Q_X(a,\gamma^{\rm adj}_*(b))
$$
for all $a\in H_{k-2j}(X)$, $b\in H_{k}(Y)$. 
\item[(ii)] We say that $\gamma$ \emph{admits an adjoint} if it admits a $k$-adjoint for all $k$.
\end{enumerate}
\end{definition}

\begin{proposition}
\label{prop:adjoint}
If the standard conjectures $B(X)$ and $B(Y)$ hold, every correspondence $\gamma\in\Corr(X,Y)$ admits an adjoint.
\end{proposition} 

\begin{proof}
Let $\gamma\in\Corr_j(X,Y)$ and consider the map 
$$
\gamma_*:H_k(X)\to H_{k+2j}(Y).
$$
As $B(X)$ and $B(Y)$ hold, the operators $s_X$ and $s_Y$ are algebraic  by Lemma \ref{sX alg}. As $s_X$ and $s_Y$ commute with the Lambda operator, we obtain
\begin{eqnarray*}
Q_Y(\gamma_*(a),b) & = & \langle\gamma_*(a),\Lambda_Y^{d_Y-k-2j}(s_Y(b))\rangle \\
& = & \langle a,{^t \gamma}_*(\Lambda_Y^{d_Y-k-2j}(s_Y(b)))\rangle \\
& = & \langle a,s_X(\Lambda_X^{d_X-k}(s_X(L_X^{d_X-k}({^t \gamma}_*(\Lambda_Y^{d_Y-k-2j}(s_Y(b))))\rangle.
\end{eqnarray*}
Hence
$$
\gamma^{\rm adj} = s_X\circ L_X^{d_X-k}\circ{^t \gamma}\circ\Lambda_Y^{d_Y-k-2j}\circ s_Y
$$
is an adjoint of $\gamma$.
\end {proof}

To use the existence of an adjoint, we need a linear algebra lemma (cf. \cite[Lemma 5]{V10},\cite[Lemma 1.6]{V4}).
\begin{lemma} 
\label{linalg}
Let $H$ and $H'$ be finite-dimensional $\QQ$-vector spaces equipped with non degenerate bilinear forms $Q:H\times H\to\QQ$ and $Q':H'\times H'\to\QQ$. Suppose that there exist linear maps
$$
\alpha:H'\to H,\ \ \beta: H\to H'
$$
such that
\begin{enumerate}
\item[(a)] $\alpha$ is surjective;
\item[(b)] $Q'\vert_{\ima(\beta\times\beta)}$ is non degenerate;
\item[(c)] $Q(\alpha(x),y) = Q'(x,\beta(y))$ for all $x\in H'$, $y\in H$. 
\end{enumerate}
Then $\alpha\comp \beta:H\to H$ is an isomorphism.
\end{lemma}

\begin{proof}
As $H$ is finite-dimensional, it suffices to show that $\ker(\alpha\comp \beta) = 0$. Suppose that $y\in\ker(\alpha\comp \beta)$. Then $\beta(y)\in\ker(\alpha)\cap\ima(\beta)$. By (c) we have
$$
0 = Q(\alpha(\beta(y)),z) = Q'(\beta(y),\beta(z))
$$
for all $z\in H$, hence $\beta(y) = 0$ by condition (b). This gives
$$
0 =  Q'(x,\beta(y)) = Q(\alpha(x),y)
$$
for all $x\in H'$ and since $\alpha$ is surjective we obtain $y=0$. 
\end{proof}

\begin{corollary}
\label{cor: iso}  Suppose that $\gamma:\Corr_j(Y,X)$ admits an adjoint. Consider the map
$\gamma_*:H_{k-2j}(Y)\to H_{k}(X)$. Then $\gamma_*\comp \gamma^{\rm adj}_*:H_k(X)\to H_k(X)$ induces an isomorphism 
$$
\gamma_*\comp \gamma^{\rm adj}_* :\ima(\gamma_*) \mapright{\sim} \ima(\gamma_*).
$$
\end{corollary}

\begin{proof}
Apply the previous Lemma  with $H' = H_{k}(X)$, $\alpha = \gamma_*$, $\beta = \gamma^{\rm adj}_*$ and $H = \ima(\gamma_*)\subseteq H_k(X)$. Condition (a) is satisfied by construction, (b) by Hodge theory (Hodge-Riemann bilinear relations) and (c) by the adjoint condition.
\end{proof}

\begin{definition}
\label{Nhat}
The \emph{modified niveau filtration} $\wh{N}^{\bullet}$ is defined by
$$
\wh{N}^j H_k(X) = \sum\ima( {\gamma_*:H_{k-2j}(Z)\to H_k(X)}),
$$
where the sum runs over all pairs $(Z,\gamma)$ such that $Z$ is smooth projective of dimension $k-2j$ and such that $\gamma\in\Corr_j(Z,X)$ admits a $k$-adjoint.
\end{definition}

We have
$$
\wh{N}^jH_k(X)\subseteq\wt{N}^jH_k(X)\subseteq N^jH_k(X).
$$
The filtrations $N^{\bullet}$ and $\wt{N}^{\bullet}$ are compatible with the action of correspondences. The filtration $\wh{N}^{\bullet}$ is compatible with correspondences that admit an adjoint.  
\begin{proposition}
\label{prop:compatibility}
Let $\gamma\in\Corr_j(X,Y)$. If $B(X)$ and $B(Y)$ hold then  we have
$\gamma_*\wh{N}^c H_k(X)\subseteq\wh{N}^{c+j}H_{k+2j}(Y)$.
\end{proposition}

\begin{proof}
There exist a smooth projective variety $Z$ and a correspondence $\lambda\in\Corr_c(Z,X)$ such that $\lambda$ admits an adjoint and
$$
\wh{N}^c H_k(X) = \ima{\lambda_*:H_{k-2c}(Z)\to H_k(X)}.
$$
We have
$$
\lambda_*\wh{N}^c H_k(X) = \ima{(\gamma\circ\lambda)_*:H_{k-2c}(Z)\to H_{k+2j}(Y)}
$$
The image is contained in $\wh{N}^{c+j}H_{k+2j}(Y)$ since $\gamma$ admits an adjoint 
by Proposition \ref{prop:adjoint} and $(\gamma\circ\lambda)^{\rm adj} = \lambda^{\rm adj}\circ\gamma^{\rm adj}$. 
\end{proof}

\section{On K\"unneth decompositions}
\begin{definition}
\label{def:refinedK}
Let $X$ be a smooth projective variety. 
\begin{enumerate}
\item[(1)]
We say that $X$ admits a {\em refined K\"unneth decomposition} if there exist correspondences $\pi_{i,j}\in\Corr_0(X,X)$ such that
\begin{itemize}
\item $\Delta_X\sim_{\rm hom}\sum_{i,j}\pi_{i,j}$
\item 
$(\pi_{i,j})_*|_{\Gr^q_N H_p(X)} = \left\{
\begin{array}{cc}
 \id & {\rm\ if \ } (p,q) = (i,j) \\ 
 0 & (p,q)\ne (i,j).
\end{array} 
\right.
$
\item $\pi_{i,j}=0$ if and only if $\Gr^j_N H_i(X) = 0$.
\end{itemize}
\item[(2)] We say that $X$ admits a \emph{refined Chow--K\"unneth decomposition} if in addition the $\pi_{i,j}$ are projectors and 
$\Delta_X\sim_{\rm rat}\sum_{i,j}\pi_{i,j}$.
\item[(3)] We say that $X$ admits a refined K\"unneth (or Chow--K\"unneth) decomposition \emph{in the strong sense} if $\pi_{i,j}$ factors  with shift $j$ through a smooth, projective variety $Z_{i,j}$ of dimension $i-2j$ for all $i$ and $j$.
\end{enumerate}
\end{definition}

\begin{remark}
By \cite[Prop. 1.4]{V4} there exists a $Q_X$-orthogonal splitting 
$$
H^*(X) = \oplus_{i,j} \Gr^j_N H_i(X).
$$
The variety $X$ admits a refined K\"unneth decomposition if this decomposition lifts to the category $\Mot_{\rm hom}(k)$ of homological motives. It admits a refined Chow--K\"unneth decomposition if the decomposition lifts to the category $\Mot_{\rm rat}(k)$ of Chow motives.
\item In an analogous way one can define refined K\"unneth (Chow--K\"unneth) decompositions with respect to the filtrations $\wt{N}^{\bullet}$ and $\wh{N}^{\bullet}$.
\end{remark}

The proof of the following result is a reformulation of the proof of \cite[Thm. 1]{V4} in terms of the modified niveau filtration.
\begin{proposition}
\label{prop:refinedK}
If $B(X)$ holds, there exists a refined K\"unneth decomposition in the strong sense with respect to the filtration $\wh{N}^{\bullet}$.
\end{proposition}

\begin{proof}
Conjecture $B(X)$ implies that the K\"unneth components are algebraic, i.e., there exist correspondences $\pi_i\in\Corr_0(X,X)$ such that $(\pi_i)_*|_{H_j(X)} = \delta_{ij}.\id$. 
By Proposition \ref{prop:compatibility} the proof of \cite[Prop. 1.4]{V4} goes through for the filtration $\wh{N}^{\bullet}$, and we obtain a $Q_X$-orthogonal splitting
$$
H^*(X) = \oplus_{i,j} \Gr^j_{\wh{N}} H_i(X).
$$
The aim is to construct correspondences $\pi_{i,j}\in\Corr_0(X,X)$ that induce this decomposition. This is done by descending induction on $j$. If $j>i/2$ we take $\pi_{i,j}=0$. Suppose that the correspondences $\pi_{i,k}$ have been constructed for $k>j$. As before there exist $Z$, smooth of dimension $i-2j$, and $\gamma\in\Corr_j(Z,X)$ such that
$$
\wh{N}^j H_i(X) = \ima(\gamma_*:H_{i-2j}(Z)\to H_i(X)).
$$
By replacing $\gamma$ with $\pi_i \comp\gamma$ if necessary, we may assume that 
$\gamma_*\mid_{H_{\ell}(Z)}=0$ if $\ell\ne i-2j$. The correspondence $\pi = \pi_i - \sum_{k>j}\pi_{i,k}$ induces the projection $\wh{N}^j H_i(X)\to\Gr^j_{\wh{N}} H_i(X)$. Put 
$\gamma'= \pi\comp\gamma$. By construction
$$
\gamma'_*:H_{i-2j}(Z)\to\Gr^j_{\wh{N}} H_i(X)
$$
is surjective. As $B(X)$ holds, $\pi$ admits an adjoint by Proposition \ref{prop:adjoint}. By definition $\gamma$ admits an adjoint, hence $\gamma'= \pi\comp\gamma$ admits an adjoint and the correspondence $T=\gamma'\comp(\gamma')^{\rm adj}$ induces an isomorphism
$$
\varphi = T_*: \Gr^j_{\wh{N}} H_i(X)\to \Gr^j_{\wh{N}} H_i(X)
$$
by Corollary \ref{cor: iso}. By the Cayley--Hamilton theorem there exists a polynomial expression $\psi = P(\varphi)$ such thay $\psi\comp\varphi = \id$. Put $U = \psi(T)$ and define $\pi_{i,j} = U\comp T$. As $T_*=\varphi$ and $U_* = \psi$ we have
\begin{eqnarray*}
(\pi_{i,j})_*\mid_{\Gr^j_{\wh{N}} H_i(X)} & = & \id \\
(\pi_{i,j})_*\mid_{\Gr^q_{\wh{N}} H_p(X)} & = & 0 \text{ if   } (p,q)\ne (i,j).
\end{eqnarray*}
By construction $\pi_{i,j}$ factors with shift $j$  through a smooth projective variety of dimension $i-2j$ and $\pi_{i,j}=0$ if and only if $\Gr^j_{\wh{N}} H_i(X) = 0$.  
\end{proof}

\begin{corollary}
If $B(X)$ holds and $H_k(X)\subseteq\wh{N}^c H_k(X)$, then there exists $\pi'_k\in\Corr_0(X,X)$ such that $\pi_k\sim_{\rm hom}\pi'_k$ and such that $\pi'_k$ factors
with shift $c$  through a smooth projective variety $Z$ as in Definition~\ref{dfn:Factoring}.
\end{corollary}

\begin{proof}
By Proposition \ref{prop:refinedK} we obtain a decomposition 
$$
\pi_k = \sum_j \pi_{k,j}.
$$
with respect to the filtration $\wh{N}^{\bullet}$. As $H_k(X)\subseteq\wh{N}^c H_k(X)$ we have
$\pi_{k,j}=0$ for all $j<c$, and the result follows.
\end{proof}

The Corollary can be generalised to the following setting.
Suppose that there exists $\pi_k\in\Corr_0(X,X)$ such that $(\pi_k)_*|_{H_{\ell}(X)}=\delta_{k\ell} \cdot \id$. If $\pi\in\Corr_0(X,X)$ satisfies
\begin{eqnarray*}
\pi\comp\pi \sim_{\rm hom} \pi \\
\pi\comp\pi_k\sim_{\rm hom}\pi_k\comp\pi\sim_{\rm hom}\pi
\end{eqnarray*} 
the motive $(X,\pi)$ is a direct factor of $(X,\pi_k)$ in $\Mot_{\rm hom}(k)$.
\begin{corollary}
\label{factorisation}
Suppose that $B(X)$ holds and that $\pi\in\Corr_0(X,X)$ is a correspondence as above. Let 
$H_{\pi}=\ima(\pi)\subseteq H_k(X)$ be the sub--Hodge structure defined by $\pi$. If $H_{\pi}\subseteq \wh{N}^c H_k(X)$ there exists a correspondence $\pi'\sim_{\rm hom}\pi$ such that $\pi'$ factors   with shift $c$  through a smooth projective variety $Z$ as in f Definition~\ref{dfn:Factoring}.
\end{corollary}
\begin{proof}
The proof of Proposition \ref{prop:refinedK} shows that we have a decomposition
$\pi_k = \sum_j \pi_{k,j}$ in $\Mot_{\rm hom}(k)$. Hence 
$$
\pi = \pi_k\comp\pi = \sum_j \pi_{k,j}\comp\pi.
$$
Suppose that there exists $j_0<c$ such that $\pi_{k,j_0}\comp\pi\ne 0$. Then there exists $x\in H_k(X)$ such that $\pi_{k,j}(\pi(x))\ne 0$. Hence $H_{\pi}\cap\ima(\pi_{k,j_0})\ne 0$. This contradicts the hypothesis $H_{\pi}\subseteq\wh{N}^c H_k(X)$ since $\pi_{k,j_0}\mid_{\wh{N}^c H_k(X)} = 0$. 
\end{proof}
This result implies a  modification of \cite[Cor. 3.4, Lemma 3.5]{Var} that we need later on.
\begin{corollary}   \label{cor:vanproj} 
  Same assumptions about $M$ and $X$. Suppose that  $H_d^{\rm var}(X)\subset \hat N^c H_d(X)$.
 Then $\pi^{\rm var}\sim_{\rm hom} \widetilde  \pi^{\rm var}$ where  $\widetilde  \pi^{\rm var} \in \Corr^0(X,X)$  factors 
 through a smooth projective variety $Z $  with shift $c$ in the sense of Definition~\ref{dfn:Factoring}.
  \end{corollary}
 
\begin{remark}\ 
\label{rmk:Comparison}
The condition $H_d(X)\subset\wh{N}^c H_d(X)$ may be replaced by 
Voisin's condition of "being parametrized by algebraic cycles of codimension $c$" \cite[Def. 0.3]{V1}.  Voisin's condition implies that
$$
\gamma_*\comp {^t \gamma}_*:H_d(X)\to H_d(X)
$$ 
is a multiple of the identity. 
Our condition implies that there exists an adjoint $\gamma^{\rm adj}$ such that $\gamma_*\comp {\gamma^{\rm adj}}_*$ is an isomorphism
 with an algebraic inverse (see Corollary \ref{cor: iso} and the proof of Proposition \ref{factorisation}). This weaker result suffices for our purposes.
\end{remark}

\begin{proposition}
\label{hat-tilde}
Suppose that $B(X)$ holds and that for every smooth projective variety $Z$ of dimension $k-2j$ the condition $B_{\ell}(Z)$ holds if $\ell\le k-2j-2$. Then $\wt{N}^j H_k(X) =\wh{N}^j H_k(X)$. 
\end{proposition}

\begin{proof}
It suffices to show that for every pair $(Z,\gamma)$ as in Definition \ref{Nhat},
$\gamma$ admits a $k$-adjoint. This follows directly from Lemma~\ref{sX alg}. 
\end{proof}

%Comparing the above with Remark~\ref{is} and Remark~\ref{rmk:Dim5}, we deduce the following.
\begin{corollary} 
\label{hat-tilde2}
We have
%\begin{enumerate}
%\item[\rm (1)] 
$\wt{N}^jH_k(X)=\wh{N}^jH_k(X)$ if $k-2j\le 3$.  
In particular, if $H_k(X) = N^{[{k\over 2}]} H_k(X) $ the filtrations $\wt{N}$ and $\wh{N}$  on $H_k(X)$ coincide with the coniveau  filtration. This is true  unconditionally 
 on $H_k(X)$, $k\le 3$. If  the conjecture $B(M)$ holds, all three filtrations are   equal on  $H_k(X)$ for  $k\le 4$.
%\item[\rm (2)]  If $\dim X\le 5$ then $X$ admits a refined K\"unneth decomposition (in the %strong sense) $\Delta_X  = \sum_{k, j}\pi_{k,j}$ with respect to the filtration $\wt{N}%^{\bullet}$. If in addition the Chow motive of $X$  is finite dimensional, $X$ admits a %refined Chow--K\"unneth decomposition (in the strong sense) with respect to $\wt{N}^{\bullet}%$. 
%such that for all $k$ and $j$, $\pi_{k,j}$ is homologous to a correspondence $\pi_{k,j}'$ that %factors through a smooth projective variety $Z_{k, j}$ of dimension $k-2j$.
%\end{enumerate}
\end{corollary} 

\begin{remark*}
The condition $B_{\ell}(Z)$ in Proposition \ref{hat-tilde} is needed to obtain an algebraic correspondence that induces $s_Z$. If $H\subset H_d(X)$ is a sub-Hodge structure such that there exists a smooth projective variety $Z$ of dimension $d-2c$ such that $H_{d-2c}^{\rm pr}(Z)\to H$ is surjective then this condition is not needed and we have 
$H\subset\wh{N}^c H_d(X)$. We present an example below.
\end{remark*}

\begin{example}
Let $X\subset\PP^{d+1}$ be a smooth hypersurface of degree $d+1$. Let $Z=F_1(X)$ be the Fano variety of lines contained in $X$. If $X$ is general then $Z$ is smooth of dimension $d-2$ and the incidence correspondence induces a surjective map (cylinder homomorphism) 
$$
\gamma_*:H_{d-2}^{\rm pr}(Z)\to H_d^{\rm pr}(X);
$$
see \cite[Thm. (5.34)]{Lewis}. Hence $H_d^{\rm pr}(X)\subset\wh{N}^1 H_d(X)$ by the previous remark.
\end{example}

Concerning the existence of a refined Chow--K\"unneth decomposition (in the strong sense) for the filtrations $N^{\bullet}$, $\wt{N}^{\bullet}$ and $\wh{N}^{\bullet}$ we have the following.
\begin{proposition}\label{prop:refinedCK}
Let $X$ be a smooth projective variety over $\C$ such that $B(X)$ holds and $h(X)$ is finite dimensional. Then
\begin{enumerate}
\item[(i)] There exists a refined Chow--K\"unneth decomposition in the strong sense for the filtration $\wh{N}^{\bullet}$.
\item[(ii)]  There exists a refined Chow--K\"unneth decomposition in the strong sense for
\begin{itemize}
\item $\wt{N}^{\bullet}$ if $\dim X\le 5$;
\item $N^{\bullet}$ if $\dim X\le 3$.
\end{itemize}
%If $\dim X\le 5$ there exists a refined Chow--K\"unneth decomposition in the strong sense for %the filtration $\wt{N}^{\bullet}$.
%\item[(iii)] If $\dim X\le 3$ there exists a refined Chow--K\"unneth decomposition in the %strong sense for the filtration $N^{\bullet}$.
\end{enumerate}
\end{proposition}

\begin{proof}
By Proposition \ref{prop:refinedK} there exists a refined K\"unneth decomposition in the strong sense for the filtration $\wh{N}^{\bullet}$. If $h(X)$ is finite--dimensional the ideal
$$
\ker A_d(X\times X)\to H_{2d}(X\times X)
$$
is nilpotent, and the refined K\"unneth decomposition lifts to $\Mot_{\rm rat}(k)$ by a lemma of Jannsen \cite{J2}. This proves part (i). Part (ii) follows from the comparison between the filtrations:
$\wt{N}^j H_i(X) = \wh{N}^j H_i(X)$ if $j-2i\le 3$ (Corollary \ref{hat-tilde2}) and $N^j H_i(X) = \wt{N}^j H_i(X)$ if $j-2i\le 1$.
\end{proof}

\begin{remark} Part (ii) is due to Vial \cite{V4}. The assumption $\dim X\le 5$ can be replaced by the conditions of Proposition \ref{hat-tilde}.
\end{remark}

\begin{remark}% \footnote{(Robert) deze opmerking toegevoegd} 
Using Proposition \ref{prop:refinedCK}, the main result of \cite{multbloch} can be extended to arbitrary dimension, provided one replaces Vial's filtration $\wt{N}^\bullet$ in the statement of \cite[Theorem 3]{multbloch} by the filtration $\wh{N}^\bullet$.

\end{remark}

 \section{The main results}  \label{sec:main}

The setup that we consider in this section is the following. Let $M$ be a smooth projective variety of dimension $d+r$. Let $L_1,\ldots,L_r$ be very ample line bundles on $M$, and let $f:\XX\to B$ denote the family of all smooth complete intersections of dimension $d$ defined by sections of $E=L_1\oplus\ldots\oplus L_r$.  We write $X_b = f^{-1}(b)$.
The next result  
plays a major role in deriving   the main results.
 It  uses the assumption that the $L_j$ are very ample in a crucial way.
\begin{proposition}[Voisin \cite{V1}]
\label{key} 
Suppose that for general $b\in B$  one has that $X_b$ has nontrivial variable homology in degree $d$.
 Let $\DD$ be a codimension-$d$ cycle on $\XX\times_B \XX$ with the property that   
    \[ 
     \DD\vert_{{X_b\times X_b}}=0\ \ \text{ in  }  H_{2d}({X_b\times X_b}).
     \]
   Then there exists  a codimension-$d$ cycle $\gamma$ on  $ M\times M$ such that
    \[   
    \DD\vert_{{X_b\times X_b}} -\gamma\vert_{X_b\times X_b} =0  \\\ \text{  in   } A_{d}({X_b\times X_b})
    \]  
    for all $b\in B$. 
    %(Here $i$ denotes the inclusion $\XX\times_B \XX\to B\times M\times M$.)
   \end{proposition}
   \begin{proof}  We want to sketch  a proof of Voisin's  original result \cite[Proposition 1.6]{V1}  since we want to point out where the
  assumptions  are used.
    Consider  the blow up $\widetilde {M\times M}$ of the diagonal and the natural quotient map $\mu: \widetilde {M\times M} \to M^{[2]}$
   to the Hilbert scheme of zero-dimensional subschemes of $M$ of   length two.
   Set $\PP=\PP H^0(X,E)$ and  as in \cite[Lemma 1.3]{V1} introduce
  \begin{equation*}
     % \label{eqn:hilb2E}
   I_2(E):= \sett{(s,y)\in \PP\times \widetilde {M\times M} }{s|_{\mu(y)}=0}.
 \end{equation*}
 Next, consider  the  blow up of $\XX\times_B\XX$ along the relative diagonal:
  \[
  p : \widetilde{\XX\times_B\XX} \to \XX\times_B\XX.
  \]
  Observe that  $\widetilde {\XX\times_B\XX}$ is Zariski-open in $I_2(E)$ and so it makes sense to restrict cycles on $I_2(E)$
  to the fibers $\widetilde{X_b\times X_b}$ of $\widetilde{\XX\times_B\XX} \to B$.
 Very ampleness of the $L_j$ implies that $I_2(E)\to \widetilde {M\times M}$ is a projective bundle and hence its 
cohomology can be expressed in terms of cohomology coming from $\widetilde{M\times M}$  and a tautological class.
Assume now that
\begin{center} 
$\exists R\in A^d(I_2(E))$ with $R\vert_{\wt{X_b\times X_b}}  \sim_{\rm hom} 0$.%\footnote{(Robert) hier tilde toegevoegd op $X_b\times X_b$}
\end{center}
 Voisin shows  that this implies the existence of a  codimension-$d$ cycle $\gamma$ on $M\times M$ and an integer
  $k$ such that  %\footnote{(Robert) hier $(p_b)_\ast$ toegevoegd}  
  %  \[
  %\aligned
   % (p_b)_\ast( R\vert_{\wt{X_b\times X_b}}) -\gamma\vert_{X_b\times X_b} &=0    \quad \text{  in   } A_{d}({X_b\times X_b}),  \\
  %k  [\Delta_{M}]   -  [\gamma] |_{X_b\times X_b}  &=0 \quad \text{ in } H_{2d}(X_b \times X_b).
  %\endaligned \]
  \[
  \aligned
   (p_b)_\ast( R\vert_{\wt{X_b\times X_b}})  = k\Delta_{X_b\times X_b} + \gamma\vert_{X_b\times X_b}    \quad \text{  in   } A_{d}({X_b\times X_b})
   \endaligned \]
 The first summand  acts on all of homology, while the second summand, by construction, acts only on the fixed homology.
 So the assumption that there is some variable homology implies that $k=0$ and so the cycle $\gamma$ is homologous to zero. To prove the above variation, suppose we are given $\DD$
  of codimension $d$ on  $\XX\times_B \XX$ as above. As $\wt{\XX\times_B \XX}\subset I_2(E)$ is Zariski open, there exists  a codimension-$d$  cycle $R$ on $I_2(E)$  such that
  $ R\vert_{\wt{\XX\times_B \XX}}=p^\ast \DD$. 
    Then we have
    \[  
    R\vert_{\wt{X_b\times X_b}}=p^\ast \DD\vert_{\wt{X_b\times X_b}}=(p_b)^\ast \bigl( \DD\vert_{X_b\times X_b}  \bigr)=0\ \text{ in  } H_{2d}(\wt{X_b\times X_b})\ \]
    for all $b\in B$, where $p_b\colon \wt{X_b\times X_b}\to X_b\times X_b$ denotes the blow-up of the diagonal. Hence, if we apply Voisin's original proposition
  to this cycle $R$,  we get the desired conclusion. 
  \end{proof}

 \begin{theorem}\label{main1} Notation as above. Suppose that $B(M)$ holds and the Chow motive of $M$    is finite-dimensional. Assume that  for a general $b\in B$  the fiber $X_b$ has non-trivial variable homology:
  \[  H_d(X_b)^{\rm var}\not=0,
  \]
  and that for some  nonnegative  integers $c,e$, with $e<d$  we have
  $$  
  H_k(X_b)=\wh{N}^c H_k(X_b)\ \ \text{for\ all\ }k \in \{e+1,\ldots,d\}.$$
  Then for any $b\in B$ 
\[ 
\text{\rm Niveau} \bigl(A_k^{}(X_b)\bigr) \le e-k  \quad \text{for\ all\ }k <\min\{d-e,c\},
\]
i.e.,   there exists a subvariety $Y_b\subset X_b$ of dimension $e$ such that
$A_k(Y_b)\to A_k(X_b)$ is surjective.%  if $k<\min\{d-e,c\}$.
  \end{theorem}
  
  \begin{proof} 
{\it Step 1.}  We first construct a   homological decomposition of the diagonal of $X_b$  
 \[ 
   \Delta_{X_b} \sim_{\rm hom}   \Delta_{\rm left}  +\Delta_{\rm mid}  +  \Delta_{\rm right}\ \ \text{in}\ H_{2d}(X_b\times X_b),
    \]
where  the right hand side are self-correspondences of  $X$ of degree $0$,     $\Delta_{\rm right}= {^t   \Delta}_{\rm left}$ and $\Delta_{\rm mid} $ factors with shift $c$
through a smooth variety $Z$. 

  This is done as follows. As conjecture $B$  is stable by hyperplane sections 
  (see Remark~\ref{Bholds}),   the complete intersections $X_b$ satisfy $B(X_b)$ and hence by Proposition~\ref{KunnIsAlg}
   there are  correspondences  $\pi_j \in \Corr^0(X_b,X_b)$, $j=0,\dots, 2d$ inducing the corresponding homological K\"unneth projectors.
    By Proposition \ref{factorisation}, for  $k  \in\{e+1,\ldots,d\}$ we have   that $ \pi_k(X_b) \sim _{\rm hom} \pi_k'(X_b)$,
  a projector  that factors through a variety  with shift $c$ as in Definition~\ref{dfn:Factoring}. Now set
  \[
  \aligned
  \Delta_{\rm left} &=\sum_{k \le e} \pi_k(X_b)\\
  \Delta_{\rm right}&= {^t}\!  \Delta_{\rm left}\\
  \Delta_{\rm mid}& =\sum_{k=e+1}^{2d-e-1} \pi'_k(X_b).
  \endaligned
  \]
  {\it Step 2. }  We   spread out  the fiberwise correspondences  $\Delta_{\rm left} , \Delta_{\rm right},  \Delta_{\rm mid}$  to the family of hypersurfaces
   \[ 
   \XX\ \to\ B\ ,
   \]
   using Voisin's argument in the form of propositions \ref{spread1} and \ref{spread2}. This gives a homological decomposition of the relative diagonal, in the sense 
   that there exist $\YY\subset\XX$ of relative dimension $d$ and  a family $\Z\to B$ of relative dimension $d-2c$, and   codimension-$d$ cycles
   \[   
   \Pi_{\rm left},\quad  \Pi_{\rm right}, \quad  \Pi_{\rm mid}
   \]
   on $\XX\times_B \XX$ such that $ \Pi_{\rm left},  \Pi_{\rm right}$  have  support on $\YY\times_B \XX$, resp. on $\XX\times_B \YY$, and $\Pi_{\rm mid}$  factors through $\Z\to B$
   such that for any $b\in B$, restriction gives back the diagonal:
   \[ 
   \left. \Bigl(\Pi_{\rm left}+\Pi_{\rm mid}+\Pi_{\rm right}\Bigr) \right\vert_{X_b\times X_b}=\Delta_{X_b}\ \ \text{in}\ H_{2d}(X_b\times X_b)\ .
   \]
 {\it Step 3. }   We upgrade this to rational equivalence  using properties of $M$. So we consider the difference
  \[ 
  \DD:= \Delta_{\XX}-\Pi_{\rm left}-\Pi_{\rm mid}-\Pi_{\rm right} ,
  \]
a relative correspondence   with the property that
    \[ 
    \DD\vert_{X_b\times X_b}=0\ \ \text{in}\ H_{2d}(X_b\times X_b)\ ,
    \]
    for all $b\in B$.  To upgrade this  to rational equivalence we applying the key Proposition~\ref{key} to $\DD$. We find a 
    codimension-$d$ cycle $\gamma$ on $M\times M$ such that
    \[
     \DD\vert_{X_b\times X_b}-\gamma\vert_{X_b\times X_b}=0\ \text{in}\ \Corr_0(X_b\times X_b)\ ,
     \]
    for all $b\in B$. The crucial point is that  the restriction  $ \gamma \vert_{X_b\times X_b} \in  \Corr_0(X_b\times X_b) $ is homologically trivial, 
    and so, by Proposition~\ref{weaknilp}   is nilpotent. 
    
    \noindent{\it Step 4. }  We can now finish the proof. Observe that   a specialization argument reduces the proof to showing
    it for a general $b\in B$.  (cf. \cite[Thm. 1.7]{V0} and \cite[Thm. 0.6]{V1}).
    For general $b$ the fibre $X_b$ will be in general position with respect to $\YY$ and $\Z$ so that
    \[ 
    \Gamma_{\rm left}:= \Pi_{\rm left}\vert_{X_b\times X_b}
    \]
    will be supported on $Y_b\times X_b$ with $Y_b$ of dimension $c$, and likewise
   \begin{equation} \label{eqn:MidFactor}
    \Gamma_{\rm mid}:=\Pi_{\rm mid}\vert_{X_b\times X_b}
     \end{equation}
    will factor with a shift $c$.  Let $ \Gamma_{\rm right}$ be the transpose of $\Gamma_{\rm left}$. For some $N\gg 0$ we have 
   \begin{equation}
   \label{eqn:DiagSplits}
     \Bigl( \Delta_{X_b} - \Gamma_{\rm left} - \Gamma_{\rm mid} - \Gamma_{\rm right}\Bigr)^{\comp  N}=0\ \ \text{in  }\Corr_0(X_b\times X_b) ,
   \end{equation}
    where $ \Gamma_{\rm left},  \Gamma_{\rm right}$ is supported on $Y_b\times X_b$, resp. on $X_b\times Y_b$, and $ \Gamma_{\rm mid}$ factors through $Z_b$ with shift $c$
     as in Eqn.~\eqref{eqn:MidFactor}.
    
  Since    $ \Gamma_{\rm left}$   is supported on $Y_b\times X_b$, Lemma~\ref{lem:Factoring} implies that its  action on $A_k(X_b)$ is trivial for $k< \codim Y= d-e$.
  The correspondence  $ \Gamma_{\rm mid}$ by construction
     factors through $Z_b$  with  shift $c$ and so -- by the same Lemma --
    its action on $A_k(X_b)$  is  trivial, since $k<c$.  
    Now expand the expression \eqref{eqn:DiagSplits} to conclude that
        \[  
    (\Delta_{X_b})_\ast  =(\text{polynomial\ in\ } \Gamma_{\rm right})_\ast\colon\ \  A_k(X_b)\to A_k(X_b) .
    \]
Since $ \Delta_{X_b}$ acts as the identity on $ A_k(X_b)$ this implies  indeed that   $A_k(X_b)$ is supported on $Y_b$, a variety of dimension $e$.
\end{proof}

\begin{remark} 
It is possible to be more precise: in the situation of Theorem \ref{main1}, we even have that
  \[ 
  \cdot L^{d-e}\colon\ \ A^{e-k}(X_b)\ \to\ A^{d-k}(X_b)
  \]
  is surjective in the range $k<\min\{d-e,c\}$, so the $k$-cycles of $X_b$ are supported on a dimension $e$ complete intersection. To obtain this, we remark that the $ \Gamma_{\rm right}$ in the above proof can be expressed in terms of $L^{d-e}$, just as in the proof of \cite{moi2}.
\end{remark}

Recall that for curves $A_0^{\rm AJ}=0$ and so, if $A_0(X_b)$ is supported on a curve, we have  $A_0^{\rm AJ}(X_b)=0$. We thus deduce that
for $c=1$, $e=1$  we get the following special case:

\begin{corollary}\label{cor}  Let $M$ be  a smooth  $(d+1)$-dimensional  projective variety 
for which $B(M)$ holds and  whose (Chow) motive  is finite-dimensional.
%  \[  H_k(M)=\wh{N}^1 H_k(M), \ \ i=2,\ldots, n-1\ .\]
 Let $X_b$, $b\in B$ be the family of all smooth hypersurfaces in a very ample linear system and suppose that
  \[  
  H_d(X_b)^{\rm var}\not=0
  \]
  and
  \[ 
  H_k(X_b)=\wh{N}^1 H_k(X_b), \ \ k=2,\ldots,d
  \]
  for the general $b\in B$. Then
  \[ 
  A_0^{AJ}(X_b)=0\ 
  \]
  for all $b\in B$.
  \end{corollary}
  
  \begin{remark}  
  (1) In view of Cor.~\ref{hat-tilde2}(1), for $n=2$ the condition on the coniveau becomes  $N^1H_2(X_b)=H_2(X_b)$, i.e. all cohomology is algebraic.
 For $n=3$ we should have in addition that $N^1H_3(X_b)=H_3(X_b)$ that is   $h^{3,0}(X_b)=0$
  as well as  the generalized Hodge conjecture for $H^3(X_b)$.\\
  (2) Note that in corollary \ref{cor}, there is no condition on $H_{d+1}(M)$, so $p_g(M)$ could be non-zero. In this case, nothing is known about the Chow groups of $M$, so it is remarkable that one can at least control the image
  \[ 
  \ima\Bigl( A_1(M)\ \to\ A_0(X_b)\Bigr)\ .
  \]
  \end{remark}

 We next come to our second main theorem.  It asserts  that a "short"  niveau filtration on the variable cohomology  already 
 has  strong implications for the Abel-Jacobi kernels.
 
\begin{theorem}
\label{main2}
Let $i: X\into  M$ be a complete intersection of dimension $d$. Suppose that
\begin{enumerate}
\item  $B(M)$ holds;
\item  The Chow motive of $M$  is finite dimensional; 
%\item \modif{ $H_k(M) = N^{[{k\over 2}]}H_kM)$ for $k\le d $;
\item $H_d^{\rm var}(X)\ne 0$ and for some positive integer $c$ we have $H_d^{\rm var}(X)\subset{\hat N}^c H_d(X)$. 
\end{enumerate}
Then   for $k<c$ or for $k>d-c$ we have 
\[
  i^*: A_{k+r}^{\rm AJ}(M) \onto A_k^{\rm AJ}(X) ,\quad 
                         i_* :A_k^{\rm AJ}(X) \into  A_k^{\rm AJ}(M ).
\]
Moreover, in this range 
\[
 A_k^{\rm var} (X)= \ker(A_k(X) \mapright{i_*}  A_k(M))=0 ,
 \]
 If in addition 
 \begin{enumerate}
\item[(a)]   $H_k(M) = {N}^{[{k\over 2}]}H_k(M)$ for $k\le d$, then  $A_k^{\rm AJ}(X) = 0$  if $k<c$ or $k>d-c$;
\item[(b)]  $H_k(M) = {N}^{[{{k+1}\over 2}]}H_k(M)$ for $k\le d$, then  $A_k^{\rm  hom}(X) = 0$  if $k<c$ or $k>d-c$.
\end{enumerate}
\end{theorem}

\proof
Let $X$ be a smooth complete intersection. In Section~\ref{sec:VarAndFixed} we showed that there is  a decomposition 
\begin{eqnarray*}
\Delta_{X} &= & 
 \pi^{\rm fix}(X)+\pi^{\rm var}(X)
\end{eqnarray*}
which in cohomology induce  projection onto fixed and variable cohomology respectively.
%\modif{The  assumptions on the niveau filtration imply that the correspondence  
%$\pi'$   is the restriction of a cycle on $M\times M$,   and that}
By Proposition~\ref{spread2} there exists relative  codimension-$d$ cycles   $\Pi'$ and  $\Pi^{\rm var} $ on $\XX\times_B \XX$ such that $\Pi'$  comes from $M\times M\times B$ and
 and $\Pi^{\rm var}$ induces  $\pi^{\rm var}(X)$. Moreover,  the restriction of
$$
R = \Delta_{\XX/B}-\Pi'-\Pi_d^{\rm var}
$$
to the general fiber is homologically trivial.  By Proposition \ref{key} there exists a  codimension-$d$ cycle $\gamma$ on $M\times M$ such that
$$
R|_{X\times X}-\gamma|_{X\times X}
$$
is rationally equivalent to zero for $b\in B$ general.
In particular $\gamma|_{X\times X}$ is homologically trivial. Hence 
$\gamma|_{X\times X}$ is nilpotent by Proposition \ref{weaknilp}.
 Let $N$ be the index of nilpotency of $\gamma|_{X\times X}$. We obtain
$$
0 = \gamma^{\comp  N}  \mid_{X\times X} = (\Delta_{X}-\pi^{\rm fix}(X) - \pi^{\rm var}(X))^{\comp  N}.
$$
By  assumption (3)  and  Corollary~\ref{cor:vanproj}      the  correspondence  $\pi^{\rm var}(X)$  factors  through a correspondence of degree $-c$ 
over a variety of dimension $d-2c$  and so acts trivially on $A_k^{\rm AJ}(X)$ if $k<c$
or $k>d-c$.  Setting $\psi=\pi^{\rm fix}(X) $,
we find that  for some polynomial $P$ we have $P(\psi)_* \comp \psi_*=\psi_*\comp P(\psi)_*=\id$  on the Chow groups  $A_k(X)$ with $k$ in this range and the 
first assertion  follows.  For the second, observe that   $ \psi$    acts as zero on
$ A_k^{\rm var} (X)$.

The  assumption (a) in the last clause   implies that $\pi^{\rm fix}(X) $ factors through a curve and so this summand acts trivially on $A_k^{\rm AJ}(X)$ for \emph{all } $k$. So then the  above argument indeed gives   that  $A_k^{\rm AJ}(X) = 0$  if $k<c$ or $k>d-c$. In case (b), $\pi^{\rm fix}(X) $ factors through a point and we obtain
$A_k^{\rm hom}(X) = 0$  if $k<c$ or $k>d-c$. 
\endproof
 
\begin{corollary}  \label{cor:fdm} In the above situation, suppose that $c =  [{d \over 2}]$. Then the motive $h(X)$ is finite-dimensional. Moreover, if for $M$ we have $ A_k^{AJ}(M)=0$ for all $k$, then   also  $A_k^{\rm AJ}(X) = 0$  for all $k$.
%In addition, if $h(M)$ is of abelian type,  then so is $h(X)$.
 \end{corollary}
\proof
The assumptions imply surjectivity of $i^*:   A_k^{AJ}(M,\id, r) \to A_k^{\rm AJ}(h(X),\id,0)$ in the range $k=0,\ldots, [{d-2 \over 2}]$. We then apply Vial's result \cite{V3}, stated 
in the Appendix as Theorem~\ref{thm: findim}. 
\endproof

\section{Variants with group actions} \label{sec:vars}

Let $M$ be a projective manifold of dimension $d+r$ and   let  $L_1,\dots, L_r$ be  ample line bundles  on $M$ 
and,   as before, set
\[
E:= L_1\oplus \cdots\oplus L_r.
\]
We assume that  a  finite group $G$ acts on    $M$ and  on the $L_j$ and that the linear systems $|L_j|^G$, $j=1,\dots,r$  are   base point free.
The complete intersection in $  M$ corresponding to  $ s= (s_1,\dots,s_r)\in  \PP(H^0(M, E))$  is denoted $X_s$.
  We consider smooth complete intersections coming from $G$-invariant hypersurfaces and set 
accordingly
\[
B:= \{     b  \in   \PP(H^0(M, E)^G)  \mid   X_b \text{ is smooth} \}. 
\]
This is Zariski open in $ \PP(H^0(M, E)^G) $.

The graph of the action of $g\in G$ on $M$ will be written $\Gamma_g\subset M\times M$.
As before, we let $\wt{M\times M}$ be the blow up of $M\times M$ in the diagonal and $M^{[2]}$ the Hilbert scheme  of length $2$ subschemes of $M$
with the natural quotient morphism 
\begin{equation*}
%\label{eqn:mu}
\mu: \wt{M\times M}\to M^{[2]}.
\end{equation*}
Consider the  "bad" locus  
\[
\aligned
B_{E,\mu} =  &\sett{y\in \wt{M\times M}} { \text{no } s \in   H^0(M, E)^G   \text{    separates the points}\\
& \hspace{12em} \text{      of the length-two scheme } \mu(y) }.
\endaligned
\] 
Note that the $G$-invariant sections of $E$ do not separate points in $G$-orbits. We demand  instead that they   separate  entire $G$-orbits; in fact we want something less stringent, as
expressed by   the following notion,
involving the proper transforms   $\wt{\Gamma_g}$ of $\Gamma_g$  in  $\wt{M\times M}$.
\begin{definition} Assume $(M,E)$ and $G$ as above. We say that $H^0(M,E)^G$ 
  \emph{almost separates} orbits if    
  the "bad" locus $B_{E,\mu}$   is contained in $\bigcup_{g\not=\id} \wt{\Gamma_g}\cup R_G$, where $R_G$ is a (possibly empty)  union of components of codimension $>\dim M=d+r$. %\modif{changed}
\end{definition}
This demand ensures that  $I_2(E) \to \wt{M\times M}$  is  a  repeated blow up of a projective bundle so that its  cohomology can be controlled.
In  order to have an  analogue  of Proposition~\ref{key},  we demand that for $g\in G$ the endomorphisms  
\[
\gamma_g^{\rm var}= [\Gamma_{g,b}]^{\rm var}_* \in   \End H_d(X_b)^{\rm var}
\]
should be independent. This can be tested using  the following result.
\begin{lemma} Let $\rho: G \to \text{GL}(V)$ be a representation of a finite group on a finite dimensional $\QQ$-vector space $V$. Then the endomorphisms $\{ \rho_g, g\in G\}$ 
are independent in $\End V$  if $G$ is abelian and every irreducible representation occurs in $V$.
\end{lemma}
\proof This is a consequence  of elementary representation theory. We may work over $\C$. In the abelian case the  group ring $\C[G]$ is isomorphic to the regular representation of $G$ and 
since the former has for its base the irreducible non-isomorphic characters, the elements $g, g\in G$ give  a basis for $\C[G]$. The representation $\rho$ induces
an algebra  homomorphism $\tilde\rho : \C[G]\to  \End V$ which is injective if every irreducible representation occurs in $V$. So the images $\tilde\rho_g$, $g\in G$ form an independent set.
\endproof

Let us next introduce some notation.
Suppose that $\chi: G\to \QQ $ is a $\QQ$-character defining an irreducible   $\QQ$-representation $V_\chi$,
i.e. $ \chi(g) =\Tr( g)|_{V_\chi}$ for all   $g\in G$. The corresponding projector in the group ring of $G$  is
\[
\pi_\chi= \frac 1 {|G|} \sum_{g\in G}  \chi(g) g \in \QQ[G]
\]
leading to 
\begin{equation}
\label{eqn:groupproj}
\Gamma_\chi:= \frac 1 {|G|}   \sum_{g\in G}  \chi(g) \Gamma_{g,b} \in \Corr_0(X_b,X_b)
\end{equation}
acting  on the Chow group of $M$ and on the homology groups of $M$ as well as the homology of the complete intersections  $X_b$.
The latter action  preserves  the decomposition into variable and fixed homology.
The $j$-th Chow group of the motive $(X, \Gamma_\chi)$ is by definition
\[
A_j(X, \Gamma_\chi)=  \im\left( \Gamma_\chi : A_j(X) \to A_j(X)\right) =   A_j(X)^\chi,
\]
where for any $G$-module $V$ we set
\[
V^\chi := \sett{ v\in V}{g (v) =\chi(g) v \text{ for all } g\in G} =  \sett{ v\in V}{ (\Gamma_\chi)_* v= v}.
\]
Thus $\Gamma_\chi$ act as the identity on $V^\chi$.  
\medskip

We are now ready to formulate a  variant of  Proposition~\ref{key}. Its validity   is shown in the course of the  proof of \cite[Theorem 3.3]{V1}.

\begin{proposition} \label{key2} Let  $(M,E)$, $G$  and $B\subset  \PP(H^0(M, E)^G)$ be as above.  Suppose  that 
\begin{enumerate}
\item  $H^0(M,E)^G$   almost separates  orbits;
\item   the endomorphisms  $\gamma_g^{\rm var}\in\End H_d(X_b)^{\rm var} $, $g\in G$ are linearly independent; 
\item  for general $b\in B$  one has $ H_d(X_b)^{\rm var}\not=0$.
\end{enumerate}
Then for   any    $\DD\in A^d(\XX\times_B \XX)^\chi$ with the property that   
    \[ 
     \DD\vert_{{X_b\times X_b}}=0\ \ \text{ in  }  H_{2d}({X_b\times X_b})^\chi,
     \]
   there exists  a codimension-$d$ cycle $\gamma$ on  $ M\times M$ such that
    \[   
    \DD\vert_{{X_b\times X_b}} -\gamma\vert_{X_b\times X_b} =0  \\\ \text{  in   } A_{d}({X_b\times X_b})^\chi
    \]  
    for all $b\in B$. 
\end{proposition}

Using this variant, the arguments we employed in  Section~\ref{sec:main} 
for $\Delta_X$ can thus  be applied to  $\Gamma_\chi$ provided we restrict  to $H_*(X_b)^\chi$.
Since  $\Gamma_\chi$   acts as the identity on $A_j(X)^\chi$,  the same  conclusions as before can be drawn for these Chow groups and we obtain  the following results.

\begin{theorem} \label{main1Bis} Let  $(M,E)$, $G$  and $B\subset\PP(H^0(M, E)^G)$ be as above.   Moreover, let $\chi$ be a    $\QQ$-character for $G$
and $\Gamma_\chi$ the associated projector \eqref{eqn:groupproj}.
Suppose that 
\begin{enumerate}
\item  $B(M)$ holds;
\item    $H^0(M,E)^G$   almost separates  orbits;
\item  the endomorphisms  $\gamma_g^{\rm var}\in\End H_d(X_b)^{\rm var} $, $g\in G$ are linearly independent; 
\item  the Chow motive $(M,\Gamma_\chi)$ is finite-dimensional.
\end{enumerate}
Assume, moreover,  that  for a general $b\in B$ one has 
$H_d(X_b)^{\rm var}\not=0$   and that
 \[ 
  H_k(X_b)^\chi \subset \wh{N}^c H_k(X_b)  \ \ \text{for\ all\ }k \in \{e+1,\ldots,d \}.
 \]
  Then for any $b\in B$ 
\[
 \text{\rm Niveau} \bigl(  (A_j^{}(X_b)) ^\chi \bigl) \le e-j\ \ \ \text{for\ all\ }j<\min\{d-e,c\},
 \]
i.e.,   there exists a subvariety $Z_b\subset X_b$ of dimension $d$ such that
$A_j(Z_b)\to A_j(X_b,\Gamma_\chi)$ is surjective if $j<\min\{d-e,c\}$.

\end{theorem}

\begin{theorem} \label{main2Bis} Notation as in the previous theorem.
Let $X\subset M$ be a $G$-invariant  complete intersection  of dimension $d$. Suppose that
\begin{enumerate}
\item  $B(M)$ holds;
\item   $H^0(M,E)^G$   almost separates  orbits;
\item   the endomorphisms  $\gamma_g^{\rm var}$, $g\in G$ are linearly independent in $\End(H_{d}(X)^{\rm var})$;
\item  the Chow motive $(M,\Gamma_\chi)$ is finite-dimensional;
%\item $H_k(M)^{\pi_\chi} = N^{[{i\over 2}]}H_k(M)^{\pi_\chi}$ for $i\le d$;
\item $ 0 \not= H_n(X)^{\rm var}$ and for some positive integer $c$ we have $H_d (X)^{\rm var, \chi} \subset{\hat N}^c H^d(X)$.  
\end{enumerate}
Then   for $k<c$ or for $k>d-c$ we have 
\[
  i^*: A_{k+r}^{\rm AJ}(M)^\chi  \onto A_k^{\rm AJ}(X)^\chi ,\quad 
                         i_* :A_k^{\rm AJ}(X)^\chi \into  A_k^{\rm AJ}(M )^\chi.
\]
Moreover, in this range 
\[
 A_k^{\rm var} (X)^\chi = \ker(A_k(X)^\chi  \mapright{i_*}  A_k(M)^\chi)=0 ,
 \]
 If in addition   $H_k(M)^\chi = {N}^{[{k\over 2}]}H_k(M)^\chi $ for $k\le d$, then  $A_k^{\rm AJ}(X)^\chi  = 0$  if $k<c$ or $k>d-c$.
\end{theorem}
 
 We also have the analogue of Corollary~\ref{cor:fdm}:
  \begin{corollary} \label{cor:fdmBis}  In the above situation, suppose that $c =  [{d \over 2}]$. Then the motive $h(X,\Gamma_\chi)$ is finite-dimensional. 
  %In addition, if $h(M,\Gamma_\chi)$ is of abelian type, then so is $h(X,\Gamma_\chi)$.
 \end{corollary}

\section{Examples}
\label{sec:exmples}

\subsection{A threefold of general type with finite  dimensional motive} 
In \cite{finmts} one of the authors  investigated a quasi-smooth  threefold $X$  which is a complete intersection of three degree $6$ hypersurfaces in the weighted projective
space  $P = \PP(2^4, 3^3)$ and showed that $A_0(X)=\QQ$. 
Let us check that this example can also be treated within the present framework. The only technical obstacle
is that $P$  and $X$ have  (mild) singularities, but -- as in  loc. cit., close inspection of the proofs shows that this does not matter. % modified Aug28

The threefold $X$ is of general type and has Hodge numbers   $h^{1,0}(X)=h^{2,0}(X)=0$, $h^{1,1}=1$,
$h^{3,0}=0$, $h^{2,0}=6$. Moreover, the intermediate jacobian $J^2(X)$ is an abelian variety and there is a curve $C$ and a correspondence $\gamma\in \Corr_1(C,X)$
inducing a surjection $J(C) \onto J^2(X)$. Hence $H^3(X)= N^1 H^3(X)$. Since $H^2(X)= N^1H^2(X)$ and $H^1(X)=0$ we can apply  Cor~\ref{cor:fdm} to conclude that
$h(\hat X)$ is finite dimensional   where $\hat X$ is a toroidal resolution of $X$.  Moreover, the cycle class map is injective in \emph{all} degrees.
%Since by theorem~\ref{main2}, we also have that $A_0^{AJ}(X)= A_1^{AJ}(C)=0$  (there is no %contribution from $P$) we can apply Vial's results 
%\cite[]{V3} to the correspondence $\gamma$ and 
%deduce that  up to Lefschetz-motives $h(\hat X)$ is  a submotive  of $h(C)$  which makes the %fact that $X$ is abelian more explicit.

\subsection{Hypersurfaces of abelian threefolds}
 \label{sec:hypab3folds}

We let $A$ be an abelian variety of dimension three. Let $\iota=-1_A$ be the
standard involution.  Choose  an irreducible  principal polarization $L$ that is preserved by $\iota$. The following facts are well known (see e.g. \cite{LB}).
\begin{facts}  
\item{$\bullet$}  $L$ is ample and  sections     of $L^{\otimes 2}$  correspond  to   even theta functions (and hence  are invariant under the involution).
\item{$\bullet$}  $L^3=3! =6 $ and  $\dim H^0(L^{\otimes 2})= 8$.
\item{$\bullet$}  The linear system  $| L^{\otimes 2}|$ defines a $2$-to-$1$ morphism $\kappa: A\to \km A\subset \PP^7=\PP H^0(L^{\otimes 2})^*$, 
where $\km A$ is the Kummer threefold associated to $A$, an algebraic threefold, smooth outside the  images of the $2^6$
two-torsion points of $A$.  

\end{facts}

We let $X=\{\theta_0=0\} \subset A$ be a general divisor in $| L^{\otimes 2}|$. This  is  a  smooth surface  invariant under $\iota$ and $\kappa$ induces
an \'etale double cover of surfaces $X\to Y=X/(\iota|_X)\subset \km A$. 
The crucial properties of $A$ are as follows. We use the standard notation   for the character spaces  for  the action of $\ZZ/2\ZZ=\{\id,\iota\}$ 
on a vector space $V$:
\[
V^\pm  = \{ v\in V\mid \iota(v)= \pm v\}.
\]
\begin{proposition}   \label{prop:hypab3fold}
\begin{enumerate}
\item 
We have  $H_1(X)^+=0$;
\item the splitting
\[
H_2^{\rm var}(X)= H_2^{\rm var, +}(X)\oplus  H_2^{\rm var, -}(X)
\]
is non-trivial and $H_2^{\rm var, +}(X)= N^1H_2^{\rm var, +}(X)$, i.e.,   $H^{2,0}(X)^{\rm var, +}=0$.
\end{enumerate}
\end{proposition}
Before giving the proof,   we  observe that Theorem~\ref{main2Bis} and Corollary~\ref{cor:fdmBis} imply:  
\begin{corollary} \label{cor:isfindimmot}
We have $A_0^{\rm var}(X)^+=0$ and the motive $h(X)^+= h(Y)$ is finite-dimensional  (of abelian type).
\end{corollary}

We now give the
\proof[Proof of Proposition~\ref{prop:hypab3fold}]
(1) Since  $\iota$ acts as $-\id$ on one-forms, $b_1(Y)= b_1(X)^+= 0$.\\
(2) We consider cohomology instead of homology. Consider the Poincar\'e residue sequence
\[
0 \to \Omega^3_A  \to \Omega^3_A(X)  \mapright{\rm res}  \Omega^2_X \to 0.
\]
In cohomology this gives
\[
0 \to H^0(\Omega^3_A) \to H^0(\Omega^3_A(X)) \mapright{\rm res}  H^0(\Omega^2_X) \to H^1(\Omega^3_A) \to 0.
\]
Since $H^0(\Omega^3_A(X))= H^0(L^{\otimes 2})$ we deduce that 
\[
h^{0,2}_{\rm var} (X) = 7,\quad h^{0,2}_{\rm fix} (X)= 3.
\]
By the residue sequence, variable holomorphic $2$-forms are the Poincar\'e-residues along $X$ of meromorphic $3$-forms 
on $A$  with at most a simple pole along $X=\{\theta_0=0\}$ are given by expressions of the form
\[
\frac{\theta}{\theta_0} dz_1\wedge dz_2\wedge dz_3
\]
with $\theta$  a  theta-function  on $A$ corresponding to a section  of $L^{\otimes 2}$, and where  $z_1,z_2,z_3$  are holomorphic coordinates on $\C^3$. It follows that
such forms are \emph{anti-invariant} under $\iota$  and so  $h^{2,0}_{\rm var}(X)=h^{2,0}_{\rm var,-}(X)=7$

To complete the proof, we need to  show that $H^{1,1}_{\rm var, +} (X)=H^{1,1}_{\rm var}(Y)$  is non-trivial.   This is a consequence of the following calculation. \endproof
 
\begin{lemma}  The invariants  of $X$ and $Y$ are as follows.
\begin{center}
\begin{tabular}{|c|c|c|c|}
\hline
{ variety}  & $b_1$ & $b_2^{\rm var}=(h^{2,0}_{\rm var},h^{1,1} _{\rm var},h^{0,2} _{\rm var} )$  & $ b_2^{\rm fix}=(h^{2,0}_{\rm fix},h^{1,1} _{\rm fix},h^{0,2} _{\rm fix} ) $ \\
\hline
$X$  & $6$ & $43=(7,29,7)$ & $15=(3,9,3)$ \\
\hline
$Y $&    $0$ & $7=(0,7,0)$ & $15=(3,9,3)$ \\
\hline
\end{tabular}
\end{center}
\end{lemma}
\proof  By Lefschetz' theorem $b_1(X)=b_1(A)=6$. 
To calculate $b_2$ we observe that $c_1(X)= -2 L|_X$ and $c_2(X)=  4 L^2|_X$ so that 
\[
c_1^2(X)=c_2(X)=4 L^2|_X=8 L^3 =48.
\]
Since    $c_2(X)= e(X)=2- 2b_1(X)+ b_2(X)=48$, it follows that $b_2(X)= 58$. Now  $b_2^{\rm fix,+}(X)= b_2(A)=15$ and so $b_2^{\rm var}(X)= 43$.
The $2$-forms on $X$ that are the restrictions of  holomorphic  $2$-forms on $A$ are 
clearly invariant  and $h^{2,0}_{\rm fix}(X)=h^{2,0}_{\rm fix,+}(X)=3$. Since $h^{2,0}_{\rm var}=7$, the invariants for $X$ follow.

For $b_2(Y)$ we use that $\iota|X$ acts freely on the generic $X$ and so 
$ e(Y)= \half e(X)= \half c_2(X)= 24 =  2 + b_2(Y)$ implying that $b_2(Y)= 22$.
 Using  K\"unneth, we find $b_2^{\rm fix, +} (X)= b_2^+(A)= b_2(A)= 15$ and so $b_2^{\rm fix,+}(X)= 15$, and $b_2^{\rm var,+}(X)= 7$. Since $h^{2,0}_{\rm var,+}(X)=0$,
 this yields the invariants for $Y$. \endproof

\subsection{Burniat-Inoue surfaces}
The preceding  example can be used  to  investigate the motive of the classical Burniat-Inoue surfaces.  By definition a  Burniat surface is a minimal surface
 $Y$  of general type with invariants  
  \[
 p_g(Y)=q(Y)=0, \quad e(Y)=6 \implies b_1(Y)=0, b_2(Y)= h^{1,1}(Y)=4.
 \]
 Such surfaces have been constructed by  Burniat in \cite{bu}, while Inoue in \cite{inoue} gave a different construction as a quotient of a hypersurface in a product of three
 elliptic curves. It is this construction that we follow.

 It has recently been shown by Pedrini-Weibel \cite[Theorem 9.1]{PW} and, independently, by Bauer-Frapporti 
 \cite{BF}  that for  such $Y$ one has $A_0(Y)=\QQ$. \footnote{In loc. cit. this is in fact  shown for  the so-called  generalized  Burniat-type surfaces with $p_g=q=0$.}  %\footnote{(Robert) Misschien nog opmerken dat Pedrini-Weibel ook onafhankelijk het Bloch vermoeden voor Burniat--Inoue oppervlakken bewijzen \cite[Theorem 9.1]{PW} ?}
We give a  different proof   fitting our set-up. The reader will notice  that our proof  is much simpler.
To explain the construction of the surface from \cite{inoue},  consider the abelian threefold
\[
A:= E_1 \times E_2 \times E_3  , \quad E_\alpha= \C/ \Lambda_\alpha, \, \text{ with }   \Lambda_\alpha=\ZZ\oplus \ZZ\tau_\alpha,\, \alpha=1,2,3.
\]
and the group $G$ generated by three commuting involutions
\[
\aligned 
\iota_1  :  (z_1,z_2,z_3)  &\mapsto   (z_1,  -z_2+\half, z_3+\half)\\
\iota_2  :  (z_1,z_2,z_3)  &\mapsto   ( z_1 +\half, z_2, -z_3+\half) \\
\iota_3  :  (z_1,z_2,z_3)  &\mapsto   (- z_1 +\half, z_2+\half, z_3).
\endaligned
\]
We recall  some classical  facts  about theta functions on an elliptic curve $E$ with period lattice generated by $1$ and $\tau \in \mathfrak h$.
The $2$-dimensional space $H^0(E, 2\cdot [0])$ is generated by two theta-functions. This space is a representation for the group $G_E$ defined as the group generated by $\iota$ 
and  the translation  $t_\half$ over the half period $[\half]$.   All of $H^0(E, 2\cdot [0])$ is invariant under $\iota$.
One can find two theta functions that are interchanged under $t_\half$ and we let $\Theta_E^+ ,\Theta_E^-$ be their sum, respectively their difference.
Then  $H^0(E, 2\cdot [0])= (++)\oplus (+-)$ as a $G_E$-module. Now set
\[
\aligned
% \Theta_E^1  &=\vartheta_1^2+\vartheta_2^2,\, \Theta_E^2 =\vartheta_1^2-\vartheta_2^2 \\
  \Theta_{j_1j_2j_3}& := \Theta_{E_1}^{j_1}  \Theta_{E_2}^{j_2} \Theta_{E_3}^{j_3}.
\endaligned
\] 
These give a basis for the space %$H^0(A, M)$ \footnote{(Robert) Misschien liever een andere letter dan $M$, die al gebruikt wordt voor de ambient space ?} 
of sections of the line bundle $ \OO_{E_1}(2\cdot[0]) \boxtimes    \OO_{E_2}(2\cdot[0]) \boxtimes  \OO_{E_3}(2\cdot[0])$
consisting of common eigenvectors for the action of $G$. Indeed,   $\C\Theta_{j_1j_2j_3}=( j_2j_3\, \, j_1j_3\, \, j_1j_2)$.

For generic $c \in \C$    the  equation  $\Theta_{+++} + c\Theta_{--}=0$  defines a $G$-invariant
 surface $X$ in $A$  on which $G$ acts freely. The quotient $X=Y/G$ is a classical Burniat-Inoue surface.
 The crucial observation  is that the involution $j= \iota_1\iota_2\iota_3$ is just the standard involution $x\mapsto -x$ on $A$. Then  Corollary~\ref{cor:isfindimmot} shows that 
the  Chow motive of
the surface $  Y/j$  is finite dimensional. This then is also true for $X=Y/G$, but since $p_g(X)=0$ it follows that
automatically  $A_0(X)=\QQ$.

It is worthwhile to note that our argument cannot be applied directly to the group $G$ since the condition that  the endomorphisms $\gamma_g, g\in G$ 
be independent,  is not fulfilled in this
case. See the table below which gives the character spaces.
\begin{center}
\begin{tabular}{|c||c||c|c|c|c||c|c|c|}
\hline
space &  $\mathbf 1$&  $-\mathbf 1 $   &   $-++$  & $+-+$ &  $++-$&  $+--$   &   $-+-$  &  $--+$  \\
\hline
$H^{0,1}(A)=H^{0,1}(X)$ & $ $   &$ $ &  $  $  &  $ $  &&   $ 1 $  &   $ 1$& $ 1$\\
\hline
$H^{0,2}(A)=H^{0,2}_{\rm fix}(X)$  & $ $  &  $ $ &   $   $  & $ $ &&   $  1$  &   $ 1$&$ 1$  \\
\hline
$H^{1,1}(A)=H^{1,1}_{\rm fix}(X)$  & $ 3$     &$ $ & $     $&  $ $ & &   $ 2 $  &   $ 2$& $ 2$\\
\hline
$H^{0,2}_{\rm var}(X)$  &   &   $  1  $   &   $ 2 $  &   $2$& $  2$ &     &   $ $&  \\
\hline
$H^{1,1}_{\rm var}(X)$  &$   1$  &   $ $   &   $  $  &   $  $& &   $   $  &   $ $&  \\
\hline
\end{tabular}
\end{center}
\begin{remark}
A variant of this argument applies   to all  generalized Inoue-Burniat surfaces, i.e. those surfaces forming  the families
$\Ss_1,\dots,\Ss_{16}$ from \cite{BCF}.    This will be treated in a forthcoming publication.
\end{remark}

%We now let $G=\ZZ/2\ZZ$ with generator $\iota$ and projector $\pi=\half( 1+\iota)$.  Since $|L^{\otimes 2}|$ is base point free and $G$-invariant, it is
%a $G$-big linear system. The motive of $A$ is finite dimensional. Since  the action of $G$ on  $H_2^{\rm var}(X)$ has an  invariant as well an anti-invariant
%subspace, the graphs of $\id_A$ and $-\id_A$ give  linearly independent endomorphisms on this space. Moreover, 
%since
%\[
%H^2_{\rm var, +}(X)  = H^{1,1}_{\rm var, +}(X) ,
%\]
%we may apply Remark~\ref{modifhyps2} to conclude that
%\[
%A_0^{\rm AJ, \rm var }(X,\Gamma_{1+\iota})=0.
%\]

\subsection{Hypersurfaces in products of  a hyperelliptic curve  and a K3-surface}
\label{sec:exmple2}

Let $C$ be a hyperelliptic curve with hyperelliptic involution $\iota_C$, and let $S$ be a K3-surface  with $h(S)$ finite dimensional
and  which admits a fixed point free involution $\iota_2$. Such surfaces exist, see e.g. the examples of Enriques surfaces  in \cite[\S 4]{bp}
coming from a K3-surface with Picard number $\ge 19$.  By remark~\ref{findinmots} the motive of $S$
-- and hence of $M:=C\times S$ -- is finite dimensional.  The involution $\iota= (\iota_1,\iota_2)$ acts without  fixed points
on $M$.
We let $L_1$ be the hyperelliptic divisor on $C$ and we pick a very ample divisor $L_2$ on $S$ invariant under the Enriques involution $\iota_2$
and we set $L=   L_1\boxtimes  L_2$.
Let \[
i: X  \into  M= C\times S
\]
be a smooth hypersurface in  $|L|$ invariant under $\iota$.  Since $\iota$ has no fixed points,   $Y=X/\iota$ is a smooth surface.
The analogues of Proposition~\ref{prop:hypab3fold}  and its corollary are valid here.
\begin{proposition}   \label{prop:hyp3fold} We have
\begin{enumerate}
\item 
$H_1(X)^+=0$;
\item  $H^{2,0}(X)^{\rm fix, +}=0$;
\item the splitting
\[
H_2^{\rm var}(X)= H_2^{\rm var, +}(X)\oplus  H_2^{\rm var, -}(X)
\]
is non-trivial and   $H^{2,0}(X)^{\rm var, +}=0$;
\item   $A_0(X)^{\rm var, +}=0$  %\modif{changed}
 and the motive $h(X)^+= h(Y)$ is finite-dimensional   of abelian type.
\end{enumerate}
\end{proposition}

\proof   To simplify notation, we write  
\[
2u=L_2^2,\quad u\in \ZZ
\]
which is possible since $L_2^2 $ is even.\\
 {\em Step 1.   Calculation of the Betti numbers of $X$ and $Y$.} \\
 We claim:
\begin{itemize}
\item $b_1(X)= 2g$ and  $b_2(X)= 4g+ 4(g+2) u  + 46 $,
\item  $ b_1(Y)=0$ and    $b_2(Y)= g+ (2g+2) u   + 22$.
\end{itemize}
To show this,  observe that  the K\"unneth formula and  the Lefschetz 
hyperplane theorem imply  $b_1(X)=b_1(M)= b_1(C)=2g$  and  $b_1(Y)= b_1^+(X)= b_1^+(C)=0$.
To calculate $b_2(X)$ we    calculate  the Euler number $e(X)=c_2(X)$ from the Whitney product  formula
\[
(1+ c_1(j^*L))( 1+ c_1(X)+ c_2(X))= 1+ (2-2g)  P_1  + 24  P_2   +\cdots, \quad P_1=   i^*p_1^* [C], \, P_2= i^*p_2^* [S] 
\]
which gives   $c_1(X)= (2-2g)P_1-  c_1(i^*L)$ and hence 
\[
\aligned
c_2(X) &=  24  P_2 -  c_1(j^*L) c_1(X)\\
&= 24 P_2 +(2g-2) P_1 \cdot  c_1(i^*L) + c_1^2(i^*L)\\
& = 24 P_2 +(2g-2) P_1 \cdot  (2P_1 + \ell_2) +  (2P_1 + \ell_2)^2, \quad  \ell_2 = c_1(  i^*p_2^* L_2) .
\endaligned
\]
Identifying $H^4(X,\ZZ)$ with the integers,  we have  
\[
P_1^2  = 0,   \hspace{2em} P_2 =  2, \hspace{2em}
(P_1  \cdot  \ell _2 )   =   L_2^2= 2u ,  \hspace{2em} \ell_2^2= 4u .  
\]
and so 
\[
c_2(X) = 48 + 4(g+2) u= 2 -4g+ b_2(X) \implies b_2(X)= 46+4g+ 4(g+2) u.
\]
We  calculate $  b_2(Y)$  from the Euler number of $Y$ as follows.
\[
\aligned
2+ b_2(Y)= e(Y) &=\half e(X) =\half (2- 2g +b_2(X)) \\
 & \implies b_2 (Y) =   \half b_2(X) -(g+1)=  22+g+ 2(g+2) u.  \hspace{4em}   
\endaligned
\]
{\em Step 2. Variable and fixed homology}.\\
Remarking  that  the fixed cohomology equals $\im ( i^*: H^2(M) \into H^2(X))$, we find   $b_2^{\rm fix}(X)= b_2(M)= b_2(C)+b_2(S)=23$. 
Since $M/\iota = \PP^1\times \{\text{Enriques surface}\}$, we find $b_2^{\rm fix}(X)= b_2(M/\iota)=  11$. 
We put the result in a table. 
\begin{center}
\begin{tabular}{|c|c|c|}
\hline
 variety  &   $b_2^{\rm var}$  & $ b_2^{\rm fix}$ \\
\hline
$X$  & $4g+  4(g+2)u   + 23  $ & $23$ \\
\hline
$Y $&      $g+ (2g+2) u   + 22$ & $11$ \\
\hline
\end{tabular}
\end{center}

\noindent {\em Step 3. Hodge numbers of $X$}.\\
As one readily verifies, the   fixed cohomology has   Hodge numbers   
\[
h^{2,0}_{\rm fix}(X)= 1,\quad h^{1,1}_{\rm fix}(X) = 21.
\]
For the variable cohomology we have  
\[
 h^{2,0}_{\rm var} (X) =  ( g+1)\cdot u +g+2,\quad h^{1,1}_{\rm var} (X) = 2(g+3)u +2g-2.
 \]
To see this  consider the  Poincar\'e residue sequence in this situation.
\[
0 \to \Omega^3_M\to \Omega^3_M(X) \mapright{\rm res} \Omega^2_X \to 0.
\]
From the long exact sequence in cohomology  we deduce that 
\begin{equation}
\Omega^3_M(X)  = p_1^* \omega_C(L_1) \wedge  p_2^* \omega_S(L_2) \implies
h^{2,0}_{\rm var} (X) =  h^0(C,\omega_C\otimes L_1)\cdot (h^0(S, L_2) - g).
\label{eqn:Var2Forms}
\end{equation}
By Riemann-Roch $ h^0(C,\omega_C\otimes L_1)= h^1(L_1^*)= g+1$ and $h^0(S,  L_2) = u+2$. The result for $h^{2,0}_{\rm var} (X)$   follows.
\\
\noindent {\em Step 4. Hodge numbers of $Y$}.\\
From the fact that $M/\iota$ is the product of $\PP^1$ and an Enriques surface, we that  find
$ h^{2,0}_{\rm fix +}= 0 $ and $ h^{1,1}_{\rm fix +}= 11 $.
To find the Hodge numbers for the variable cohomology,
we use a basic observation.
\begin{lemma} \label{lem:killforms} We have $h^0(C,\omega_C\otimes L_1)^+=0$.
\end{lemma}
\proof
Invariant  meromorphic $1$-forms on $C$ having a pole at most in the hyperelliptic divisor correspond to meromorphic $1$-forms on $\PP^1$ with at most $1$ pole.
But there are no such forms.
\endproof
As a corollary, from    \eqref{eqn:Var2Forms}   it  then follows that $h^{0,2}_{\rm var} (X)^+=0$ and so  $H^2(X)_{\rm var}^+$ is pure of type $(1,1)$. 
We claim that  $H^2(X)_{\rm var}^+\not=0$.  Indeed, our calculations lead to the following table.
\begin{center}
\begin{tabular}{|c|c|c|}
\hline
 variety  &   $(h^{2,0}_{\rm var},h^{1,1} _{\rm var},h^{0,2} _{\rm var} )$  & $(h^{2,0}_{\rm fix},h^{1,1} _{\rm fix},h^{0,2} _{\rm fix} ) $ \\
\hline
$X$  & $( (g+1)u + g+2, 2(g+3)u+ 2g +21, g+1)u + g+2 )$ & $(1,21,1)$ \\
\hline
$Y $&      $(0, 2(g+3)u + 2g +10,0)$ & $(0,11,0)$ \\
\hline
\end{tabular}
\end{center}
\endproof

\subsection{Hypersurfaces in   products of  three curves}
 \label{sec:hyppcurves}

Let $M= C_1\times C_2\times C_3$ where $C_\alpha$ are curves  equipped with an involution $\iota_\alpha$.  
Assume that $L_\alpha$ is a very ample  line bundle  on $C_\alpha$ which is preserved
by $\iota_\alpha$ and such that the system $|L_\alpha|^{\iota_\alpha}$ gives a morphism.
Put $\iota = (\iota_1,\iota_2,\iota_3)$ and 
 let $X\subset M$ be a general member of the system $|L_1 \otimes L_2\otimes L_3| ^\iota $ where we identify $L_\alpha$ with its pull back to $M$. 
 The group $G$ generated by the three involutions $\iota_\alpha$ acts on $M$.
As in the previous subsections, one can  calculate the various character spaces for the action of $G$ on  $H_2(X)^{\rm var}$.   Suppose one factor, say $C_1$,  is hyperelliptic. Using Lemma~\ref{lem:killforms}, one sees that this  makes the niveau of $H_2(X)^{\iota_1,\rm var}$
equal to $1$. Choosing the other factors suitably so that all character spaces appear in $H_2(X)^{\rm var}$ one finds (many)  projectors $\pi$
with  $A_0^{\rm AJ, \rm var }(X,\Gamma_{\pi})=0$. Let us give one concrete example. 

We let $C_1$ be a genus $g$ hyperelliptic curve,  and $C_2, C_3$    genus $3$ unramified double covers of  some  genus $2$ curve.
We take for $L_1$ the degree $2$ hyperelliptic bundle and we take for $L_\alpha$, $\alpha=2,3$ the  degree $2$ bundles for which the system
 $|L_\alpha|$ induces the unramified double cover of $C_\alpha$ onto the genus $2$ curve.
 Note that $\iota$ acts  without fixed points  in this case. As before, we let $Y=X/\iota$.
We find the following invariants.
\begin{center}
\begin{tabular}{|c|c|c|c|}
\hline
 variety  & $b_1$ & $ (h^{2,0}_{\rm var},h^{1,1} _{\rm var},h^{0,2} _{\rm var} )$  & $ (h^{2,0}_{\rm fix},h^{1,1} _{\rm fix},h^{0,2} _{\rm fix} ) $ \\
\hline
$X$  & $ 2(g+6)$ & $ (7g+16 , 14g +47 7g+16 )$ &  $(6g+9,12g +21,6g+9)$ \\
\hline
$Y $&    $8$ & $ (0,12g +28 ,0)$ & $ (4,8,4)$ \\
\hline
\end{tabular}
\end{center}

 Concluding, $H^2_{\rm var,+}(X)$ is pure of type $(1,1)$ and $H^2_{\rm var}(X)$  contains an invariant and anti-invariant part so that we can
 apply  our considerations  to the motive $(X, \half(1+\iota))$ and hence 
 \[
A_0^{\rm AJ, \rm var }(X)^+=0.
\]
It follows, as before, that $h(Y)=h(X)^+$ is finite-dimensional.\begin{remark} Using 
 \cite{multbloch} we have that 
the map
  \[A_1^{\rm hom}(Y)\otimes A_1^{\rm hom}(Y)\to A_0^{\rm AJ}(Y)\]  induced by intersection product is surjective, like in the case of an Abelian variety of dimension $2$.
To see this, consider the commutative diagram
\[
\xymatrix{ 
H^{1,0}(Y)    \otimes   H^{1,0}(Y)  \ar[d]^\simeq  \ar[r]_{\wedge} & H^{2,0}(Y)   \ar[d]^\simeq  \\
\bigotimes^2 \left(H^{1,0}(C_2/\iota_2) \oplus  H^{1,0}(C_3/\iota_3) \right)   \ar@{->>}[r]   & H^{1,0}(C_2/\iota_2) \otimes H^{1,0}(C_3/\iota_3),
}  
\]
which shows that the top-line is a surjection.
\end{remark}
%\footnote{(Robert) Misschien nog expliciet zeggen dat hier ook weer $h(Y)=h(X)^+$ eindig-dimensionaal is ?}
%\footnote{(Robert) Is het in dit voorbeeld het geval dat $H^{1,0}(Y)\otimes H^{1,0}(Y)\to H^{2,0}(Y)$ surjectief is ? Als ja, dan volgt (mbv eindig-dimensionaliteit, cf. \cite{multbloch}) dat ook
%$A^1_{hom}(Y)\otimes A^1_{hom}(Y)\to A^2_{AJ}(Y)$ surjectief is, i.e. $Y$ ``gedraagt zich als een abels oppervlak''...}

\subsection{Odd-dimensional complete intersections of four quadrics}

The following example is due to Bardelli \cite{Bardelli}. Let $\iota:\PP^7\to\PP^7$ be the involution defined by 
$$
\iota(x_0:\ldots:x_3:y_0:\ldots:y_3) = (x_0:\ldots:x_3:-y_0:\ldots:-y_3).
$$ 
Let $X = V(Q_0,\ldots,Q_3)$ be the intersection of four $\iota$--invariant quadrics. Then $H^{3,0}(X)^- = 0$, hence $H^3(X)^-$ is a Hodge structure of level one. Bardelli showed that there exist a smooth curve $C$ and a correspondence $\gamma\in\Corr_1(C,X)$ such that $\gamma_*:H_1(C)\to H_3(X)^-$ is surjective. Hence $H_3(X)^-\subseteq\wt{N}^1 H_3(X)=\wh{N}^1 H_3(X)$. By Theorem \ref{main2Bis} we get $A_0^{\rm AJ}(X)^- = 0$.

Consider the projector $p={1\over 2}(\id_X-\iota^*)$. As $\iota_* = \iota^*$, we have ${^t}p = p$. Hence the motive $N = (X,p)$ satisfies $N\cong N^{\vee}(3)$ and we can apply Theorem \ref{thm: findim} to the map $i^*:M=(\PP^7,{1\over 2}(\id_{\PP}-\iota^*))\to N$. This shows that the motive $N=h(X)^-$ is finite dimensional; more precisely, it is a direct factor of $M'=M\oplus M^{\vee}(3)\oplus_i h(C_i)(i)$ for some curves $C_i$. As $A_i^{\rm AJ}(M')=0$ for all $i$, we obtain that 
$$
A_i^{\rm AJ}(X)^-=0
$$
for all $i$. In other words  the quotient morphism $f\colon X\to Y:=X/\iota$ induces an isomorphism %\footnote{(Robert) deze zin toegevoegd},
  \[ f^\ast\colon\ \ A^\ast_{\rm AJ}(Y)\ \xrightarrow{\cong}\ A^\ast_{\rm AJ}(X)\ .\]

This example can be generalised to higher dimension. 
\begin{theorem}
Let $\iota$ be the involution on $\PP^{2m+3}$ ($m\ge 2$) defined by
$$
\iota(x_0:\cdots:x_{m+1}:y_0:\cdots:y_{m+1}) = (x_0:\cdots:x_{m+1}:-y_0:\cdots:-y_{m+1})
$$
and let $X = V(Q_0,\ldots,Q_3)$ be a complete intersection of four $\iota$--invariant quadrics. 
Let $G = \{\id,\iota\}$ and let $\chi:G\to\{\pm 1\}$ be the character defined by $\chi(\iota) = (-1)^{m-1}$. Then $H^{2m-1}(X)^{\chi}$ is a Hodge structure of level one, and there exist a smooth curve $C$ and a correspondence $\gamma\in\Corr_{m-1}(C,X)$ such that $\gamma_*:H_1(C)\to H_{2m-1}(X)^{\chi}$ is surjective. 
\end{theorem}

\begin{proof} 
See \cite[Chapter 3]{N1} or \cite[Chapter 4]{N2}.
\end{proof}

\begin{corollary} 
The motive $h(X)^{\chi}$ is finite dimensional and $A_i^{\rm AJ}(X)^{\chi}=0$ for all $i$.
\end{corollary}

\begin{remark} The same reasoning can be applied to the examples in \cite{V92}.
\end{remark}

\appendix 
\section{A variant of Voisin's arguments} \label{sec:Appendix}

\begin{proposition}
 \label{spread1}  
 Let $\Gamma$ be a   codimension-$k$ cycle on   $\XX\times_B \XX$  and suppose that for $b\in B$ very general,
     \[ 
     \Gamma\vert_{X_b\times X_b}\ \text{  in } H^{2k}(X_b\times X_b)
     \]
     is supported on $V_b\times W_b$, with $V_b, W_b\subset X_b$ closed of codimension $c_1$ resp. $c_2$.
    Then there exist closed $\VV, \WW\subset \XX$ of codimension $c_1$ resp. $c_2$, and a 
    codimension-$k$ cycle $\Gamma'$ on $\XX\times_B \XX$ supported on $\VV\times_B \WW$ and such that
     \[ 
     \Gamma^\prime\vert_{X_b\times X_b}=\Gamma\vert_{X_b\times X_b}\ \text{  in  }H^{2k}(X_b\times X_b)
     \]
    for all $b\in B$.
\end{proposition}

\begin{proof} Use the same Hilbert schemes argument as in \cite[Proposition 3.7]{V0}, which is the case $V_b=W_b$.
\end{proof}

\begin{proposition}
\label{spread2}
Suppose that $H_k(X_b) = \Nhat^c H_k(X_b)$ for all $k  \in\{e+1,\ldots,d\}$ and all $b\in B$. Then there exist families 
$\Z_k \to B$ of relative dimension $k-2c$ and relative degree zero correspondences $\Pi'_k\in\Corr_B(\XX,\XX)$ 
such that 
\begin{enumerate}
\item[(a)] $\Pi'_k$ factors through $\Z_k$;
\item[(b)] $\Pi'_k\vert_{X_b\times X_b}$ is homologous to the $k$-th K\"unneth projector $\pi_k(X_b)$ for $k=e+1,\ldots,d$.
\end{enumerate}
\end{proposition}

\begin{proof}
Using the assumptions and a Hilbert scheme argument as in \cite{V1} there exist a Zariski open subset $U\subset B$, a finite \'etale covering $\pi:V\to U$, a family $\Z_k\to V$ of relative dimension $i-2c$ and relative correspondences $\Gamma\in\Corr_V(\Z_k,\XX)$, $\Gamma'\in\Corr_V(\XX,\Z_k)$ such that 
$$
(*)\ \ Q(\Gamma_v(x),y) = Q'(x,\Gamma'_v(y))
$$
for all $x\in H_k(X_{\pi(v)})$, $y\in H_{k-  2c}(Z_{\pi(v)})$ and $v\in V$.  We now consider $\Gamma$ and $\Gamma'$ as relative cycles over $U$. Let $u\in U$. 
If   $\pi^{-1}(u) =  \{v_1,\ldots, v_N\}$ we
have  $\Gamma_u = \sum_j  \Gamma_{v_j}$, $\Gamma'_u = \sum_j \Gamma'_{v_j}$. As condition $(*)$ holds for all $v_j$,  we obtain
$$
(*)\ \ Q(\Gamma_u(x),y) = Q'(x,\Gamma'_u(y)).
$$
We can extend $\Z$ to $B$ by relative projective completion and desingularisation, and extend $\Gamma$ and $\Gamma'$ to relative correspondences over $B$ by taking their Zariski closure. 

As before, let  $H_k^{\rm fix}(X_b)$ be the image of the restriction map $H_{k+2r}(M)\to H_k(X_b)$.  As $B(M)$ holds there exists an algebraic cycle $\beta_{d+r-k}$ that induces the operator $\Lambda^{d+r-k}$. Set $R_k = \beta_{d+r-k}\comp  L^{d-k}\comp \pi_{k+2r}(M)$. If we pull back these cycles to $M\times M\times B$ and then to $\XX\times_B\XX$, we obtain  relative correspondences $\Pi_k\in\Corr_B(\XX,\XX)$ such that $\Pi_k\vert_{H_k^{\rm fix}(X_b)}$ is the identity for all $k$ % 
(see e.g. ~\cite[Lemmas 3.2 and 3.3]{Var}). 
Note that by construction $R_k$ factors through a subvariety of dimension $r+k$ of $M$ and $\Pi_k\vert_{X_b\times X_b}$ factors through a subvariety $Y_b\subset X_b$ of dimension $k$, i.e., 
$\Pi_k\vert_{X_b\times X_b}\in \ima{A_d(Y_b\times X_b)\to A_d(X_b\times X_b)}$. 

Write $\TT = \Gamma\comp \Gamma'\in\Corr_B(\XX,\XX)$. Replacing $\TT$ by $\Pi_k\comp \TT$ if necessary, we may assume that $\TT\vert_{X_b\times X_b}$ acts as zero on $H_j(X_b)$ for all $j\ne k$. By construction $(\TT_b)_*:H_k(X_b)\to H_k(X_b)$ is an isomorphism, hence it has an algebraic inverse by the Cayley-Hamilton theorem as we saw in the proof of Proposition \ref{prop:refinedK}. We want to perform a relative version of this construction. To this end, note that since $f:\XX\to B$ is a smooth morphism, the sheaf $R_kf_*\QQ$ is locally constant. Hence there exists an open covering $\{U_{\alpha}\}$ of $B$ and isomorphisms $f_{\alpha}$ from $R_kf_*\QQ\vert_{U_{\alpha}}$ to the constant sheaf with fiber $H_k(X_0)$ ($0\in U_{\alpha}$ a base point). As $\TT$ is a relative correspondence defined over $B$, the maps 
$(\TT\vert{U_{\alpha}})_*:R_kf_*\QQ\vert_{U_{\alpha}}\to R_kf_*\QQ\vert_{U_{\alpha}}$ induce automorphisms 
$$T_{\alpha}: H_k(X_0)\to H_k(X_0)
$$ 
that commute with the transition functions $f_{\alpha\beta} = f_{\alpha}\comp  f_{\beta}^{-1}$:
$$
T_ {\alpha} = f_{\alpha\beta}\comp  T_{\beta}\comp  f_{\alpha\beta}^{-1}.
$$
Hence 
%$T_{\alpha}-\lambda}I =  f_{\alpha\beta}(T_{\beta}-\lambdaI} f_{\alpha\beta}^{-1}
%$$
the characteristic polynomial of $T_{\alpha}$ does not depend on $\alpha$. This implies that there exists a polynomial $P(\lambda)$ such that 
$$
P(\TT_b)_* = (\TT_b)_*^{-1}
$$
for all $b\in B$. Define $\UU = P(\TT)\in\Corr_B(\XX,\XX)$ and set $\Pi'_k = \UU\comp \TT$. 
\end{proof}

\begin{corollary}
\label{spread CK}
There exists relative correspondences $\Pi_{\rm left}$, $\Pi_{\rm mid}$ and $\Pi_{\rm right}$ and  families $\YY\to B$ of relative dimension $d$, $\Z\to B$ of relative dimension $d-2c$ such that 
\begin{enumerate}
\item $\Pi_{\rm left}$ is supported on $\YY\times_B \XX$ and $\Pi_{\rm right}$ is supported on $\XX\times\YY$;
\item $\Pi_{\rm mid}$ factors through $\Z$;
\item
The restriction of
$$
\Delta_{\XX/B} - \Pi_{\rm left}-\Pi_{\rm mid}-\Pi_{\rm right}
$$
to $X_b\times X_b$ is homologous to zero for all $b\in B$.
\end{enumerate}
\end{corollary}

\begin{proof}
Define $\Pi_{\rm left} = \sum_{k=0}^e \Pi_k$, $\Pi_{\rm mid} = \Pi'_d + \sum_{k=e+1}^{d-1} (\Pi'_k + {^t{\Pi'_k}})$ and $\Pi_{\rm right} = {^t{\Pi_{\ell}}}$. For the support condition on $\Pi_{\ell}$ and $\Pi_r$ use Proposition \ref{spread1}.
\end{proof}

\section{On a result of Vial}

In this appendix we give a quick proof of a result of Vial \cite{V3} using the work of Kahn--Sujatha \cite{KS} on birational motives. We work with the category of covariant motives $\Mot_{\rm rat}(k)$. The Lefschetz object in this category is 
$\Lef = ({\rm Spec}(k),\id,1)$. The category $\Mot^0_{\rm rat}(k)$ of birational motives is the pseudo--abelian completion of the quotient ${\Mot_{\rm rat}(k)/\LL}$, where 
$\LL$ is the ideal of morphisms that factor through an object of the form $M\otimes\Lef$ with $M\in\Mot_{\rm rat}(k)^{\rm eff}$. We denote the image of a motive $M$ under the functor 
$$
\Mot_{\rm rat}(k)\to\Mot^0_{\rm rat}(k)
$$
by $M^0$. Kahn--Sujatha prove that 
$$
\Hom_{\Mot^0}(h(X)^0,h(Y)^0)\cong A_0(Y_{k(X)})\otimes\QQ.
$$
More generally we have \cite{Vial10}
$$
\Hom_{\Mot^0}(h(X)^0,M^0)\cong A_0(M_{k(X)})\otimes\QQ.
$$
We shall also use the category $\Mot_{\rm num}(k)$ of numerical motives, which is abelian and semisimple \cite{J1}. The image of $M\in\Mot_{\rm rat}(k)$ under the functor $\Mot_{\rm rat}(k)\to\Mot_{\rm num}(k)$ is denoted $\overline{M}$.

\begin{lemma}
\label{lemma: left inverse}
Let $f:M\to N$ be a morphism in $\Mot_{\rm rat}(k)$ such that $M$ is finite dimensional. If $\overline{f}:\overline{M}\to \overline{N}$ admits a left inverse then $f$ admits a left inverse. 
\end{lemma}

\begin{proof}
If $\overline{g}\comp\overline{f}=\id_{\overline{M}}$ then $g\comp f-\id_{M}$ is nilpotent. Writing out the expression $(g\comp f - \id_{M})^N=0$ we obtain a left inverse for $f$. 
\end{proof}

\begin{lemma}
\label{lemma1}
Let $f:M\to N$ be a morphism in $\Mot_{\rm rat}(k)$. If $M$ is finite dimensional, there exists a decomposition $N\cong N_1\oplus N_2$ such that 
\begin{enumerate}
\item $N_1$ is isomorphic to a direct factor $M_1$ of $M$ (hence finite dimensional);
\item $\overline{N}_1\cong\ima{\overline{f}}$.
\item The composition $M\to N\to N_2$ is numerically trivial. 
\end{enumerate}
\end{lemma}

\begin{proof}
In $\Mot_{\rm num}(k)$ we have decompositions $\overline{M}\cong\overline{M}_1\oplus\overline{M}_2$ and $\overline{N}\cong\overline{N}_1\oplus\overline{N}_2$ such that 
$\overline{M}_1\to\overline{N}_1$ is an isomorphism and the remaining maps 
$\overline{M}_i\to\overline{N}_j$ are zero. Since $M$ is finite dimensional, the direct summand $\overline{M}_1$ lifts to a direct summand $M_1$ of $M$. Put $\alpha = f\vert_{M_1}:M_1\to N$. As $\overline{\alpha}$ is a monomorphism it admits a left inverse. By Lemma \ref{lemma: left inverse} there exists $\beta:N\to M_1$ such that $\beta\comp\alpha = \id_{M_1}$. Define $\pi=\alpha\comp\beta\in\End(N_1)$. Then $\pi$ is a projector and we have $N=N_1\oplus N_2$ with $N_1 = (N,\pi)$ and $N_2 = (N,\id-\pi)$. Then $M_1\cong N_1$ and by construction $\overline{N}_1\cong\ima{\overline{f}}$ and $\overline{M}\to\overline{N}_2$ is the zero map.   
\end{proof}

\begin{lemma}
\label{lemma2}
Let $M = (X,p,m)$ and $N=(Y,q)$, and let $f:M\to N$ be a morphism in $\Mot_{\rm rat}(k)$ such that
\begin{enumerate}
\item $M$ is finite dimensional;
\item $A_0(M_{\Omega})\to A_0(N_{\Omega})$ is surjective, with $\Omega\supset k$ a universal domain;
\item $f$ is numerically trivial.
\end{enumerate}
Then $N^0 = 0$ in $\Mot_{\rm rat}^0(k)$.
\end{lemma}

\begin{proof}
The second assumption implies that
\begin{eqnarray*}
A_0(M_{k(Y)}) & \to & A_0(N_{k(Y)}) \\ 
\Vert & & \Vert \\
\Hom(h(Y)^0,M) & \to & \Hom(h(Y)^0,N)
\end{eqnarray*}
is surjective, hence there exists $\varphi\in\Hom(h(Y)^0,M^0)$ such that $f^0\comp\varphi = q^0 = \id_{N^0}$. In particular, $f^0:M^0\to N^0$ is an epimorphism. Write $\varphi=p^0\comp\psi$ with $\psi:h(Y)^0\to h(X)^0$. There exists $\gamma\in\Corr_{-n}(Y,X)$ such that $\gamma^0 = \psi$. Put $g=p\comp\gamma\comp q:N\to M$ and $\pi=g\comp f\in\End(M)$. As $\overline{f}=0$, $\overline{\pi}=0$. Hence $\pi$ is nilpotent since $M$ is finite dimensional. By construction $f^0\comp g^0 = \id_{N^0}$, so $\pi^0$ is a projector and by nilpotence we get $\pi^0=0$. This implies that $f^0 = f^0\comp\pi^0 =0$, hence $N^0=0$ since $f^0:M^0\to N^0$ is an epimorphism.   
\end{proof}

\begin{remark} Suppose $k=\C$.
It suffices to assume that $A_0^{\rm AJ}(M_{})\to A_0^{\rm AJ}(N_{})$ is surjective. Indeed, there exists a curve $C$ such that $J(C)\to{\rm Alb}(N)$ is surjective. We then replace $M$ by $M'=M\oplus h(C)$ and apply the Lemma to $M'$.  
\end{remark}

\begin{corollary}
\label{cor:fd}
Let $f:M=(X,p,m)\to N=(Y,q)$ be a morphism in $\Mot_{\rm rat}(k)$ such that 
\begin{enumerate}
\item $M$ is finite dimensional;
\item $A_i(M_{\Omega})\to A_i(N_{\Omega})$ is surjective for all $i\le\ell-1$.
\end{enumerate}
Then $N\cong N_1\oplus N_2$ with $N_1$ finite dimensional and $N_2\cong (Z,\rho,\ell)$ with $\dim Z = d-\ell$.
\end{corollary}

\begin{proof}
By Lemma \ref{lemma1} $N\cong N_1\oplus N_2$ with $M\to N_2$ numerically trivial, hence $N_2^0 =0$ by Lemma \ref{lemma2}. This implies that $N_2\cong R(1)$ with $R = (Z,\rho)$ and $\dim Z=d-1$. This finishes the proof if $\ell=1$. The general case follows by induction on $\ell$ using the formula $A_i(R(k)) = A_{i-k}(R)$.
\end{proof}

\begin{remark} Assume $k=\C$.
\begin{enumerate}
\item As noted before, it suffices to assume that
$$
A_i^{\rm AJ}(M_{\Omega})\to A_i^{\rm AJ}(N_{\Omega})
$$ is surjective for all $i\le\ell-1$ (here $\Omega=\C$ considered as universal domain).
\item
If the motive $M$ is self--dual up to twist, i.e., $M\cong M^{\vee}(d)$, the statement of the Corollary can be improved. Write $N = N_1\oplus R(\ell)$ as before, and consider the map
$M\cong M^{\vee}(d)\to R(\ell)^{\vee}(d) = R^{\vee}(d-\ell) = (Z,{^t \rho})=R'$. By assumption $A_i(M_{\Omega})\to A_i(R'_{\Omega})$ is surjective for all $i\le\ell-1$, hence $R'\cong R_1'\oplus R_2'$ such that $R'_1$ is finite dimensional and $R'_2 = (Z',\rho',2\ell-d)$ with $\dim(Z')=\dim Z-\ell = d-2\ell$. 
\end{enumerate}
\end{remark}

Summarizing, we get the following result.
\begin{theorem}[Vial]
\label{thm: findim} 
Let $f:M=(X,p,m)\to N=(Y,q)$ be a morphism in $\Mot_{\rm rat}(\C)$ such that $M$ is finite dimensional. 
\begin{enumerate}
\item
If
$A_i^{\rm AJ}(M_{})\to A_i^{\rm AJ}(N_{})$ is surjective for all $i\le d-1$ then $N$ is isomorphic to a direct factor of $M\oplus\oplus_{i=1}^{d} h(C_i)(i)$ where $C_i$ is a smooth curve for all $i$. 
\item
If
$M\cong M^{\vee}(d)$ and
$A_i^{\rm AJ}(M_{})\to A_i^{\rm AJ}(N_{})$ is surjective for all $i\le{{d-2}\over 2}$ then $N$ is isomorphic to a direct factor of $M\oplus M^{\vee}(d)\oplus\oplus_i h(C_i)(i)$ with $C_i$ smooth curves.
\end{enumerate}
Hence $N$ is finite dimensional in both cases.
\end{theorem}

\begin{proof}
Use Corollary \ref{cor:fd} and the previous Remark.
\end{proof}

\begin{remark}
The proof of Theorem \ref{thm: findim} gives a bit more: if the motive $M$ is ``of abelian type" (i.e., belongs to the subcategory of $\Mot_{\rm rat}(\C)$ generated by the motives of abelian varieties over $k$) then $N$ is of abelian type.
\end{remark}


\vskip 1cm
\begin{thebibliography}{dlPG99}

\bibitem{An} Y. Andr\'e, Motifs de dimension finie (d'apr\`es S.-I. Kimura, P. O'Sullivan,...), S\'eminaire Bourbaki 2003/2004, {\it Ast\'erisque} {\bf 299},  Exp. No. 929, viii, 115--145.

\bibitem{Ay} J. Ayoub, Motives and algebraic cycles: a selection of conjectures and open questions, preprint available from {\tt user.math.uzh.ch/ayoub/}.

\bibitem{Bardelli} F. Bardelli, On Grothendieck's generalized Hodge conjecture for a family of threefolds with trivial canonical bundle. {\it J. Reine Angew. Math.} {\bf 422} (1991), 165–200. 

\bibitem{bp} W. Barth, C. Peters,  Automorphisms of Enriques surfaces,  {\it Invent. Math.} {\bf 73} (1983), 383--411.

\bibitem{BCF} I. Bauer, F. Catanese and D. Frapporti, Generalized Burniat type surfaces and Bagnera--deFranchis varieties, {\tt arXiv:1409.1285v2}.

\bibitem{BF} I. Bauer and D. Frapporti,
Bloch's conjecture for generalized Burniat type surfaces with $p_g=0$.  
{\it Rend. Circ. Mat. Palermo (2}) {\bf 64} (2015),  27--42.
%\bibitem{BPV} W. Barth, K. Hulek, C. Peters and A. van de Ven, Compact complex surfaces, Springer Verlag Berlin Heidelberg New York 2004, 

%\bibitem{BCP} I. Bauer, F. Catanese and R. Pignatelli, Surfaces of general type with geometric genus zero: a survey, in: Complex and differential geometry, Hannover 2009 (W. Ebeling et alii, eds.), Springer Proceedings in Mathematics 8, Springer Berlin Heidelberg 2011,

%\bibitem{B} S. Bloch, Lectures on algebraic cycles, Duke Univ. Press Durham 1980,

%\bibitem{B2} S. Bloch, Algebraic cycles and higher K-theory, Advances in Math. vol. 61 (1986), 267--304,

%\bibitem{B3} S. Bloch, The moving lemma for higher Chow groups, J. Alg. Geom. 3 (1994), 537--568,


\bibitem{BO} S. Bloch and A. Ogus, Gersten's conjecture and the homology of schemes, {\it Ann. Sci. Ecole Norm. Sup}. {\bf 4} (1974), 181--202.

\bibitem{BS} S. Bloch and V. Srinivas, Remarks on correspondences and algebraic cycles, {\it American Journal of Mathematics}  {\bf 10} (1983), 1235--1253.

%\bibitem{BV} A. Beauville and C. Voisin, On the Chow ring of a $K3$ surface, J. Alg. Geom. 13 no. 3 (2004), 417--426,

%\bibitem{Br} M. Brion, Log homogeneous varieties, in: {\it Actas del XVI Coloquio Latinoamericano de Algebra, 
%Revista Matem\'atica Iberoamericana, Madrid} (2007),
%{\tt arXiv: math/0609669}.

\bibitem{bu} P. Burniat: Sur les surfaces de genre $P_{12}>1$. {\it Ann. Mat. Pura Appl.} {\bf 71} (1966) 1--24.    

\bibitem{CM} 
M. de Cataldo and L. Migliorini, The Chow groups and the motive of the Hilbert scheme of points on a
surface, {\it Journal of Algebra} {\bf 251} (2002), 824--848.

\bibitem{Ch} F. Charles, Remarks on the Lefschetz standard conjecture and hyperk\"ahler varieties. {\it Comment. Math. Helv.} {\bf 88} (2013),  449-468. 

\bibitem{ChM} F. Charles and E. Markman, The standard conjectures for holomorphic symplectic varieties deformation equivalent to Hilbert schemes of $K3$ surfaces, 
{\it Comp. Math.} {\bf 149}  (2013), 481--494.

%\bibitem{Cat} F. Catanese, Surfaces with $K^2=p_g=1$ and their period mapping, in:  Algebraic geometry (Copenhagen, 1978), Springer Lecture Notes in Mathematics, Springer 1979,

%\bibitem{CD} F. Catanese and O. Debarre, Surfaces with $K^2=2$, $p_g=1$, $q=0$, J. reine u. angew. Math. 395 (1989), 1--55,

%\bibitem{CH} A. Corti and M. Hanamura, Motivic decomposition and intersection Chow groups, I, Duke
%Math. J. 103 (2000), 459--522,

%\bibitem{Del} C. Delorme, Espaces projectifs anisotropes, Bull. Soc. Math. France 103 (1975), 203--223,


%\bibitem{DM} C. Deninger \and J. Murre, Motivic decomposition of abelian schemes and the Fourier transform. J. reine u.
%angew. Math. 422 (1991), 201--219,

%\bibitem{Dol} I. Dolgachev, Weighted projective varieties, in: Group actions and vector fields, Vancouver 1981, Springer Lecture Notes in Mathematics 956, Springer Berlin Heidelberg New York 1982,

%\bibitem{F} W. Fulton, Intersection theory, Springer-Verlag Ergebnisse der Mathematik, Berlin Heidelberg New York Tokyo 1984,

%\bibitem{FMSS} W. Fulton, R. MacPherson, F. Sottile, and B. Sturmfels, Intersection theory on spherical
%varieties, J. Alg. Geom. 4 (1995), 181--193,

%\bibitem{GHM} B. Gordon, M. Hanamura and J. Murre, Relative Chow-K\"unneth projectors for modular
%varieties, J. reine u. angew. Math. 558 (2003), 1--14,

\bibitem{Lie Fu} Lie Fu, On the coniveau of certain sub-Hodge structures. 
{\it Math. Res. Lett.} {\bf 19} (2012) 1097--1116. 


\bibitem{GP} V. Guletski\u{\i} and C. Pedrini, The Chow motive of the Godeaux surface, in:
{\it Algebraic Geometry, a volume in memory of Paolo Francia (M.C. Beltrametti},
F. Catanese, C. Ciliberto, A. Lanteri and C. Pedrini, editors),
Walter de Gruyter, Berlin New York, (2002).

\bibitem{inoue} M. Inoue,  
Some new surfaces of general type.
{\it Tokyo J. Math.} {\bf 17} (1994), 295--319.

\bibitem{Iv} F. Ivorra, Finite dimensional motives and applications (following S.-I. Kimura, P. O'Sullivan and others), 
available from {\tt https://perso.univ-rennes1.fr/florian.ivorra/}
        
                
%\bibitem{Iy} J. Iyer, Murre's conjectures and explicit Chow-K\"unneth projectors for varieties with a nef tangent bundle, {\it Transactions of the Amer. Math. Soc.} {\bf 361} (2008), 1667--1681.

%\bibitem{Iy2} J. Iyer, Absolute Chow-K\"unneth decomposition for rational homogeneous bundles and for log homogeneous varieties, {\it Michigan Math. Journal}
%{\bf 60},   (2011), 79--91.


%\bibitem{J} U. Jannsen, Mixed motives and algebraic K-theory, Springer Lecture Notes in Mathematics 1400 (1990),

\bibitem{J1} U. Jannsen, 
Motives, numerical equivalence, and semi-simplicity, Invent. Math. 107(3) (1992), 447--452, 

\bibitem{J2} U. Jannsen, Motivic sheaves and filtrations on Chow groups, in: {\it Motives} (U. Jannsen et alii, eds.), Proceedings of Symposia in Pure Mathematics Vol. 55 (1994), Part 1,  

\bibitem{J4} U. Jannsen, On finite-dimensional motives and Murre's conjecture, in: {\it Algebraic cycles and motives} (J. Nagel and C. Peters, eds.), Cambridge University Press, Cambridge (2007).

%\bibitem{KMP} B. Kahn, J. P. Murre and C. Pedrini, On the transcendental part of the motive of a surface, in: {\it Algebraic cycles and motives. Vol. 2},   
%{\it London Math. Soc. Lecture Note Ser.}, {\bf 344}  143--202, Cambridge Univ. Press, Cambridge (2007).

\bibitem{KS} B. Kahn and R. Sujatha, Birational motives I: Pure birational motives. {\it Ann. K-Theory} {\bf 1} (2016), 379--€"440. 

\bibitem{Kim} S. Kimura, Chow groups are finite dimensional, in some sense,
{\it Math. Ann.}  {\bf 331} (2005), 173--201.

\bibitem{K0} S. Kleiman, Algebraic cycles and the Weil conjectures, in: {\it Dix expos\'es sur la cohomologie des sch\'emas}, 359--386, North-Holland Amsterdam, (1968). 

\bibitem{K} S. Kleiman, The standard conjectures, in: {\it Motives} (U. Jannsen et alii, eds.), {\it Proceedings of Symposia in Pure Mathematics Part 1}. {\bf 55}, 
Amer. Math. Soc., Providence (1994).

\bibitem{LB}  A. Lange and C. Birkenhake, {\em Complex abelian varieties}, Springer-Verlag Berlin Heidelberg New York (1994).
1992.
%\bibitem{KS} M. Kuga and I. Satake, Abelian varieties attached to polarized $K3$ surfaces, Math. Ann. 169 (1967), 239--242,


%\bibitem{Kunev} V. Kunev, An example of a simply-connected surface of general type for which the local Torelli theorem does not hold, C. R. Ac. Bulg. Sc. 30 (1977), 323--325,



\bibitem{small} R. Laterveer, Algebraic varieties with small Chow groups, {\it J. Math. Kyoto Univ.}  {\bf 38}   (1998), 673--694.

%\bibitem{yet} R. Laterveer, Yet another version of Mumford's theorem, Archiv Math. (2015),   

\bibitem{Var} R. Laterveer, Variations on a theorem of Voisin, submitted.

\bibitem{multbloch} R. Laterveer, On a multiplicative version of Bloch's conjecture, Beitr\"age zur Algebra und Geometrie {\bf 57} (4) (2016), 723--734.

\bibitem{moi2} R. Laterveer, A brief note concerning hard Lefschetz for Chow groups, {\it Canadian Math. Bulletin} {\bf 59} (2016), 144--158.

\bibitem{fanocubic} R. Laterveer, A remark on the motive of the Fano variety of lines of a cubic, Ann. Math. Qu\'ebec 41 no. 1 (2017), 141---154.

\bibitem{Lat} R. Laterveer, Algebraic cycles on Fano varieties of some cubics, Results in Mathematics {\bf 72} no. 1 (2017), 595--616.


\bibitem{Lewis}  J. Lewis, The cylinder correspondence for hypersurfaces of degree $n$ in $\mathbf{P}^n$. {\it Amer. J. Math.} {\bf 110} (1988), no. 1, 77--114.

%\bibitem{M} D. Mumford, Rational equivalence of $0$-cycles on surfaces, J. Math. Kyoto Univ. Vol. 9 No 2 (1969), 195--204,

\bibitem{Mur} J. Murre, On a conjectural filtration on the Chow groups of an algebraic variety, parts I and II, {\it Indag. Math.}  {\bf 4} (1993), 177--201.

\bibitem{MNP} J. Murre, J. Nagel and C. Peters, {\it Lectures on the theory of pure motives},  {\it University Lecture Series} {\bf 61}, Amer. Math. Soc., Providence (2013).

\bibitem{N1} J. Nagel, The  Abel--Jacobi map for complete intersections, PhD thesis, Univeristy of Leiden (1997), http://nagel49.perso.math.cnrs.fr/thesis.pdf

\bibitem{N2} J. Nagel, Cohomology of quadric bundles, Habilitation thesis, Universit\'e Lille 1 (2006), http://nagel49.perso.math.cnrs.fr/habilitation.pdf

%\bibitem{NS} J. Nagel and M. Saito, Relative Chow-K\"unneth decompositions for conic bundles and Prym varieties, Int. Math. Res. Not. 2009, no. 16, 2978--3001,

%\bibitem{PP} A. del Padrone and C. Pedrini, Derived categories of coherent sheaves and motives of K3 surfaces, in: Regulators (J. Gil et alii, eds.), Contemp. Math. 571, Amer. Math. Soc., Providence RI, 2012,

\bibitem{connect} Paranjape, K., Cohomological and cycle-theoretic connectivity.
A{\it nn. of Math. (2)}  {\bf 139} (1994),  641--660.


\bibitem{PW} C. Pedrini and C. Weibel, Some surfaces of general type for which Bloch's conjecture holds, to appear in: {\it Period Domains, Algebraic Cycles, and Arithmetic}, Cambridge Univ. Press, (2015).

\bibitem{P} C. Pedrini, On the finite dimensionality of a $K3$ surface, {\it Manuscripta Mathematica} {\bf 138} (2012), 59--72.

%\bibitem{Ped} C. Pedrini, Bloch's conjecture and valences of correspondences for $K3$ surfaces, arXiv:1510.05832v1,

%\bibitem{PS} C. Peters and J. Steenbrink, Mixed Hodge structures, Springer-Verlag Ergebnisse der Mathematik, Berlin Heidelberg New York 2008,


\bibitem{finmts}  C. Peters, Bloch-type conjectures and an example of a three-fold of general type, {\it Communications in Contemporary Mathematics} {\bf 12}  (2010) 587--605.


%\bibitem{Rito} C. Rito, A note on Todorov surfaces, Osaka J. Math. 46 no. 3 (2009), 685--693,

%\bibitem{R} A.A. Rojtman, The torsion of the group of 0-cycles modulo rational equivalence, Annals of Mathematics 111 (1980), 553--569,

\bibitem{Tan} S. Tankeev, On the standard conjecture of Lefschetz type for complex projective threefolds. II, {\it Izvestiya Math}. {\bf 75}  (2011), 1047--1062.

\bibitem{Vial10} C. Vial, Pure motives with representable Chow groups, {\it Comptes Rendus de l'Acad\'emie des Sciences} {\bf 348} (2010), 1191--1195.

\bibitem{V} C. Vial, Algebraic cycles and fibrations, {\it Documenta Math.}  {\bf 18} (2013), 1521--1553.

\bibitem{V2} C. Vial, Projectors on the intermediate algebraic Jacobians, {\it New York J. Math}. {\bf 19} (2013), 793--822.

\bibitem{V3} C. Vial,  
Remarks on motives of abelian type. 
{\it Tohoku Math. J.} {\bf 69} (2017), no. 2, 195--€"220. 

\bibitem{V4} C. Vial, Niveau and coniveau filtrations on cohomology groups and Chow groups, {\it Proceedings of the LMS} {\bf 106}  (2013), 410--444,

\bibitem{V5} C. Vial, Chow-K\"unneth decomposition for $3$- and $4$-folds fibred by varieties with trivial Chow group of zero-cycles, {\it J. Alg. Geom.}  {\bf 24}
 (2015), 51--80.

%\bibitem{V9} C. Voisin, Remarks on zero-cycles on self-products of varieties, in: Moduli of vector bundles, Proceedings of the Taniguchi Congress  (M. Maruyama,  ed.), Marcel Dekker New York Basel Hong Kong 1994,

\bibitem{V10} C. Voisin, Remarks on filtrations on Chow groups and the Bloch conjecture,
{\it Annali di matematica pura ed applicata} {\bf 183} (2004), 421--438.

%\bibitem{V11} C. Voisin, Symplectic involutions of $K3$ surfaces act trivially on $CH_0$, Documenta Math. 17 (2012), 851--860,

\bibitem{V92}  C. Voisin, Sur les z\'ero-cycles de certaines hypersurfaces munies d'un automorphisme. {\it Ann. Scuola Norm. Sup. Pisa Cl. Sci.} {\bf 19} (1992), 473–492.

\bibitem{V0} C. Voisin, The generalized Hodge and Bloch conjectures are equivalent for general complete intersections, {\it Ann. Sci. \'Ecole Norm. Sup.}  {\bf 46} 
(2013), 449--475,

\bibitem{V1} C. Voisin, The generalized Hodge and Bloch conjectures are equivalent for general complete intersections, II, {\it J. Math. Sci. Univ. Tokyo}  {\bf 22} (2015), 491--517.

\bibitem{V8} C. Voisin, Bloch's conjecture for Catanese and Barlow surfaces, {\it J. Differential Geometry}  {\bf 97} (2014), 149--175.

%\bibitem{Vo} C. Voisin, Chow Rings, Decomposition of the Diagonal, and the Topology of Families, Princeton University Press, Princeton and Oxford, 2014,

\bibitem{Weil} A.  Weil, {\it Introduction \`a  l'\'etude des vari\'et\'es k\"ahl\'eriennes}. Publications de l'Institut de Math\'ematique de l'Universit\'e de Nancago, VI. Actualit\'es Sci. Ind. {\bf 1267}Hermann, Paris (1958) .

\bibitem{Xu} Z. Xu, Algebraic cycles on a generalized Kummer variety, {\tt arXiv:1506.04297v1}.

\end{thebibliography}
\end{document}